\documentclass[a4paper,english]{article}

\usepackage[latin9]{inputenc}
\usepackage{geometry}
\usepackage{pgfplots}
\usepackage{float}
\usepackage[caption=false]{subfig}
\usepackage{amsthm}
\usepackage{amsmath}
\usepackage{amstext}
\usepackage{amssymb}
\usepackage{color}
\usepackage{graphicx}
\usepackage{nicefrac}
\usepackage{fancyhdr}
\usepackage{lastpage}
\usepackage[english]{babel}
\usepackage{slashbox}
\usepackage{tikz}
\usepackage{graphicx,epstopdf}
\usepackage{xcolor}
\usepackage{rotating}
\usetikzlibrary{arrows,chains,matrix,positioning,scopes,shapes,arrows,calc,decorations.pathmorphing}
\usepackage{blkarray}
\usepackage{multirow}
 \usepackage{booktabs}
\usepackage{algpseudocode}
 \usepackage{algorithm}
\usepackage{hyperref}

\usepackage{listings} 
\usepackage{tikz}
\usepackage{ifthen}
\usepackage{url}

\usepackage[font={small,it}]{caption}

\usepackage{amsthm}
\usepackage{thmtools}

\newcommand{\smb}{\left[\begin{smallmatrix}}
\newcommand{\sme}{\end{smallmatrix}\right]}

\newtheorem{theorem}{Theorem}

\newtheorem{example}[theorem]{Example}
 
\newtheorem{lemma}[theorem]{Lemma}

\DeclareMathOperator{\rank}{rank}
\DeclareMathOperator{\h}{\mathcal{H}}

\DeclareMathOperator{\diag}{diag}

\DeclareMathOperator{\gap}{\mathsf{gap}}

\DeclareMathOperator{\sign}{sign}

\DeclareMathOperator{\stp}{stop}

\DeclareMathOperator{\trace}{trace}

\DeclareMathOperator{\range}{Range}

\newcommand{\R}{{\mathbb R}}
\newcommand{\N}{{\mathbb N}}

\newcommand{\Matlab}{{\sc Matlab}}

\newcommand{\calO}{{\mathcal O}}

\tikzstyle{block} = [rectangle, draw, fill=blue!20, 
    text width=5em, text centered, rounded corners, minimum height=4em]
    
\tikzset{
    ncbar angle/.initial=90,
    ncbar/.style={
        to path=(\tikztostart)
        -- ($(\tikztostart)!#1!\pgfkeysvalueof{/tikz/ncbar angle}:(\tikztotarget)$)
        -- ($(\tikztotarget)!($(\tikztostart)!#1!\pgfkeysvalueof{/tikz/ncbar angle}:(\tikztotarget)$)!\pgfkeysvalueof{/tikz/ncbar angle}:(\tikztostart)$)
        -- (\tikztotarget)
    },
    ncbar/.default=0.5cm,
}    
\tikzset{round left paren/.style={ncbar=0.5cm,out=100,in=-100}}
\tikzset{round right paren/.style={ncbar=0.5cm,out=80,in=-80}}

\title{A fast spectral divide-and-conquer method for banded matrices}
\author{Ana 
\v{S}u\v{s}njara\thanks{MATH-ANCHP, \'{E}cole Polytechnique F\'{e}d\'{e}rale de Lausanne, Station 8, 1015 Lausanne, Switzerland. E-mail: ana.susnjara@epfl.ch. The work of Ana \v{S}u\v{s}njara
has been supported by the SNSF research project \emph{Low-rank updates of matrix functions and fast eigenvalue solvers.} } \and Daniel Kressner\thanks{MATH-ANCHP, \'{E}cole Polytechnique F\'{e}d\'{e}rale de Lausanne, Station 8, 1015 Lausanne, Switzerland. E-mail: daniel.kressner@epfl.ch.} }

\begin{document}

\date{}

\maketitle

\begin{abstract}
Based on the spectral divide-and-conquer algorithm by Nakatsukasa and Higham [SIAM J. Sci. Comput., 35(3):A1325--A1349, 2013], we propose a new algorithm for computing all the eigenvalues and
eigenvectors of a symmetric banded matrix. For this purpose, we combine our previous work on the fast computation of spectral projectors in the so called HODLR format, with a 
novel technique for extracting a basis for the range of such a HODLR matrix. The numerical experiments demonstrate that our algorithm exhibits
quasilinear complexity and allows for conveniently dealing with large-scale matrices. 
\end{abstract}

%
%
%

\section{Introduction}
\label{sec:introduction}

Given a large symmetric banded matrix $A\in \mathbb{R}^{n\times n}$, we consider the computation of its \textit{complete} spectral decomposition 
\begin{equation}
\label{eq:sp_decomposition}
A = Q\Lambda Q^T,\quad \Lambda = \diag(\lambda_1,\lambda_2,\ldots, \lambda_n),
\end{equation}
where $\lambda_i, i = 1,\ldots,n$ are the eigenvalues of $A$ and the columns of the orthogonal matrix $Q\in \R^{n\times n}$ the corresponding eigenvectors. This problem has attracted quite some attention from the early days of numerical linear algebra until today, particularly when $A$ is a a tridiagonal matrix. 

A number of applications give rise to banded eigenvalue problems. For example, they constitute a critical step in solvers 
for \emph{general} dense symmetric eigenvalue problems. Nearly all existing approaches, with the notable exception of~\cite{NakaHigh2013}, first reduce a given dense symmetric matrix to tridiagonal form. This is 
followed by a method for determining the spectral decomposition of a tridiagonal matrix, such as the QR algorithm, the classical divide-and-conquer (D\&C) method or the algorithm of
multiple  relatively  robust  representations  (MRRR). All these methods have complexity $\calO(n^2)$ or higher; simply because all $n$ eigenvectors are computed and stored explicitly. 

On a modern computing architecture with a memory hierarchy, it turns out   
to be advantageous to perform the tridiagonalization based on successive band reduction~\cite{BiscLangSun2000}, with a symmetric 
banded matrix as an intermediate step~\cite{AuckBlumBungHuck2011,BienIgualKressPet2011,HaidLtaiDong2011,HaidSolcGates2013,SoloBallDemmHoef016}. In this context, it would be preferable 
to design an eigenvalue solver that works directly with banded matrices, therefore avoiding the reduction from banded to tridiagonal form. Such a possibility has been explored for 
classical D\&C in~\cite{Arbenz1992,HaidLtaiDong2012}. However, the proposed 
methods seem to suffer from numerical instability or an unsatisfactory complexity growth as the bandwidth increases.  

In this paper we propose a new and fast approach to computing the spectral decomposition of a symmetric banded matrix. This is based on the spectral D\&C method from~\cite{NakaHigh2013}, which recursively 
splits the spectrum using invariant subspaces extracted from spectral projectors associated with roughly half of the spectrum. In previous work~\cite{KressnerSus2017}, we have developed a fast method for 
approximating such spectral projectors in a hierarchical low-rank format, the so called HODLR (hierarchically off-diagonal low-rank) 
format~\cite{Ambikasaran2013}.
 However, the extraction of the invariant subspace, requires to determine a basis for the range of the spectral projector. This represents a major challenge. We present an efficient algorithm for computing an orthonormal basis of an invariant subspace in the HODLR format, which heavily exploits
properties of spectral projectors. The matrix of eigenvectors is stored implicitly, via orthonormal factors, where each factor is an orthonormal basis for an invariant subspace. Our approach extends to general symmetric HODLR matrices. 

Several existing approaches that use hierarchical low-rank formats for the fast solution of eigenvalue problems are based on computing (inexact) LDL$^{T}$ decompositions in such a format, see~\cite[sec. 13.5]{Hackbusch2015} for an overview. These decompositions allow to slice the spectrum of a symmetric matrix into smaller chunks and are particularly well suited when only the eigenvalues and a few eigenvectors are needed.

To the best our knowledge, the only existing fast methods suitable for the complete spectral decomposition of a large symmetric matrix are based on variations
the classical D\&C method by Cuppen for a symmetric tridiagonal matrix~\cite{Cuppen1980/81}. One recursion of the method divides, after a rank-one perturbation, 
the matrix into a $2\times 2$ block diagonal matrix. In the conquer phase 
the rank-one perturbation is incorporated by solving a secular equation for the eigenvalues and applying a Cauchy-like matrix to the matrix of eigenvectors. Gu and Eisenstat~\cite{Gu1995} not only 
stabilized Cuppen's method but also observed that the use of the fast multipole method for the Cauchy-like matrix multiplication reduced its complexity to $\calO(n^2)$ for computing all
eigenvectors. Vogel et al.~\cite{VogelXiaCauBal2016} extended these ideas beyond tridiagonal matrices, to general symmetric HSS (hierarchically semiseparable) matrices. Moreover, by representing
the matrix of eigenvectors in factored form, the overall cost reduces to $\calO(n \log^2 n)$. While our work bears similarities with~\cite{VogelXiaCauBal2016}, such as the storage of eigenvectors 
in factored form, it differs in several key aspects. First, our developments use the HODLR format while~\cite{VogelXiaCauBal2016} uses the HSS format. The later format 
stores the low-rank factors of off-diagonal blocks in a nested manner and thus reduces the memory requirements by a factor $\log n$ \emph{if} the involved ranks stay on the same level. However, one may 
need to work with rather large values of $n$ in order to gain significant computational savings from working with HSS instead of HODLR. A second major difference is that the spectral D\&C method used in 
this paper has, despite the similarity in name, little in common with Cuppen's D\&C. One advantage of using spectral D\&C is that it conveniently allows to compute only parts of 
the spectrum. A third major difference is that~\cite{VogelXiaCauBal2016} incorporates a perturbation of rank $r>1$, as it is needed to process matrices of bandwidth larger than one 
by sequentially splitting it up into $r$ rank-one perturbations. The method presented in this paper processes higher 
ranks directly, avoiding the need for splitting and leveraging the performance of level 3 BLAS operations. While the timings reported in~\cite{VogelXiaCauBal2016} cover 
matrices of size up to $10\,240$ and appear to be comparable with the timings presented in this paper, our experiments additionally demonstrate that our newly proposed method
 allows for conveniently dealing with large-scale matrices.

The rest of the paper is organized as follows. In section~\ref{sec:dc_sp}, we recall the spectral divide-and-conquer algorithm for computing the spectral decomposition of a symmetric matrix. Section~\ref{sec:hodlr} 
gives a brief overview of the HODLR format and of a fast method for computing spectral projectors of HODLR matrices. In section~\ref{sec:invariant_subspace} we discuss 
the fast extraction of invariant subspaces from a spectral projector given in the HODLR format. Section~\ref{sec:dc_hodlr} presents the overall spectral D\&C algorithm in the HODLR format for computing the 
spectral decomposition of a banded matrix. Numerical experiments are presented in section~\ref{sec:experiments}.

\section{Spectral divide-and-conquer}
\label{sec:dc_sp}

In this section we recall the spectral D\&C method by Nakatsukasa and Higham~\cite{NakaHigh2013} for a symmetric $n\times n$ matrix $A$ with spectral decomposition~\eqref{eq:sp_decomposition}. We assume that the eigenvalues are sorted in ascending order and choose a shift $\mu \in \R$ such that
\[
\lambda_1\leq \cdots\leq\lambda_{\nu} < \mu <\lambda_{\nu+1}\leq\cdots\leq \lambda_{n}, \qquad \nu \approx n/2.
\]
The relative spectral gap associated with this splitting of eigenvalues is defined as
$${\rm gap} = \frac{\lambda_{\nu+1}  - \lambda_{\nu}}{\lambda_n - \lambda_1}.$$ 
The spectral projector associated with the first $\nu$ eigenvalues is the orthogonal projector onto the subspace spanned by the corresponding eigenvectors. Given~\eqref{eq:sp_decomposition}, it takes the form
\[
 \Pi_{<\mu} = Q\begin{bmatrix}
                    I_{\nu} &0\\
                    0 & 0
                   \end{bmatrix}Q^{T}.
\]
Note that
\[
 \Pi_{<\mu}^{T} = \Pi_{<\mu}^2 = \Pi_{<\mu}, \qquad \trace(\Pi_{<\mu}) = \rank(\Pi_{<\mu}) = \nu.
\]
The spectral projector associated with the other $n-\nu$ eigenvalues is given by
\[
 \Pi_{>\mu} = Q\begin{bmatrix}
                    0 &0\\
                    0 & I_{n-\nu}
                   \end{bmatrix}Q^{T}
\]
and satisifies analogous properties.

The method from~\cite{NakaHigh2013} first computes the matrix sign function
$$ \sign(A-\mu I) =  Q\begin{bmatrix}
-I_{\nu} &0\\
0 &I_{n-\nu}
\end{bmatrix}Q^{T}$$ and then extracts the spectral projectors via the relations
\[
 \Pi_{<\mu} = \frac{1}{2}(I - \sign(A-\mu I)), \qquad  \Pi_{>\mu}  = I-\Pi_{<\mu}.
\]
The ranges of these spectral projector are invariant subspaces of $A-\mu I$ and, in turn, of $A$.
Letting $Q_{<\mu} \in \R^{n\times \nu}$ and $Q_{>\mu}\in \R^{n\times (n-\nu)}$ denote arbitrary orthonormal bases
for $\range(\Pi_{<\mu})$ and $\range(\Pi_{>\mu})$, respectively, we therefore obtain
\begin{equation}
 \label{eq:block_eig}
 \begin{bmatrix}Q_{<\mu}&Q_{>\mu}\end{bmatrix}^{T}A\begin{bmatrix}Q_{<\mu}&Q_{>\mu}\end{bmatrix} = \begin{bmatrix}
                                                          A_{<\mu} &0\\
                                                          0 &A_{>\mu}
                                                         \end{bmatrix},
\end{equation}
where the eigenvalues of $A_{<\mu} = Q_{<\mu}^{T}AQ_{<\mu}$ are $\lambda_1,\ldots,\lambda_\nu$ and the eigenvalues of $A_{>\mu} =Q_{>\mu}^{T}AQ_{>\mu}$ are $\lambda_{\nu+1},\ldots,\lambda_n$. Applying the 
described procedure recursively to $A_{<\mu}$ and $A_{>\mu}$ leads to Algorithm~\ref{alg:sdc_eig}. When the size of the matrix is below a user-prescribed minimal size $n_{\stp}$, the recursion 
is stopped and a standard method for computing spectral decompositions is used, denoted by {\tt eig}.

\begin{algorithm}[h!]
    \caption{\text{Spectral D\&C method}}
     \label{alg:sdc_eig}
    \textbf{Input:} Symmetric matrix $A\in \R^{n\times n}$. \\
    \textbf{Output:} Spectral decomposition $A = Q\Lambda Q^T$. 

   \begin{algorithmic}[1]
   \Function {$[Q,\Lambda] =\ $\tt sdc}{$A$}
   
   \If{$n \le n_{\operatorname{stop}}$}
   \State Return $[Q,\Lambda] = {\tt eig}(A)$.
   \Else
   
   \State Choose shift $\mu$. \label{shift}
   \State Compute sign function of $A-\mu I$ and extract spectral projectors $\Pi_{<\mu}$ and $\Pi_{>\mu}$. \label{sp}
   \State Compute orthonormal bases $Q_{<\mu}, Q_{>\mu}$ of $\range(\Pi_{<\mu})$, $\range(\Pi_{>\mu})$.\label{onb}
   \State Compute $A_{<\mu} = Q_{<\mu}^{T}AQ_{<\mu}$ and $A_{>\mu} = Q_{>\mu}^{T}AQ_{>\mu}$. \label{dc}
   \State Call recursively $[Q_1,\Lambda_1] = {\tt sdc}(A_{<\mu})$ and $[Q_2,\Lambda_2] = {\tt sdc}(A_{>\mu})$.
   \State Set $Q \gets \begin{bmatrix} Q_{<\mu} Q_1 & Q_{>\mu} Q_2 \end{bmatrix}$, $\Lambda = \begin{bmatrix} \Lambda_1 & 0 \\ 0 & \Lambda_2 \end{bmatrix}$.
   
   \EndIf
   \EndFunction
 \end{algorithmic}
\end{algorithm}

In the following sections, we discuss how Algorithm~\ref{alg:sdc_eig} can be implemented efficiently in the HODLR format.

\section{Computation of spectral projectors in HODLR format}
\label{sec:hodlr}

In this section, we briefly recall the HODLR format and the algorithm from~\cite{KressnerSus2017} for computing spectral projectors in the HODLR format. 

\subsection{HODLR format}

Given an $n\times m$ matrix $M$ let us consider a block matrix partitioning of the form
\begin{equation} \label{eq:toplevelpartitioning}
 M =\left[ \begin{array}{c|c}
   M^{(1)}_{11} & M^{(1)}_{12} \\ \hline
   M^{(1)}_{21} & M^{(1)}_{22}\\ 
     \end{array} \right].
\end{equation}
This partitioning is applied recursively, $p$ times, to the diagonal blocks $M^{(1)}_{11}$, $M^{(1)}_{22}$, leading to the hierarchical partitioning shown in Figure~\ref{fig:h_hodlr_matrices}. We say that $M$ is a \emph{HODLR matrix} of level $p$ and HODLR rank $k$ if all off-diagonal blocks seen during this process have rank at most $k$. In the HODLR format, these blocks are stored, more efficiently, in terms of their low-rank factors. For example, for $p=2$, the HODLR format takes the form
\begin{equation*} \label{eq:toplevelpartitioning2}
 M = \left[ \begin{array}{c|c} \small
      \begin{array}{c|c}
   M^{(2)}_{11} & U_1^{(2)} \big( V_2^{(2)} \big)^T  \\ \hline
   U_2^{(2)} \big( V_1^{(2)} \big)^T & M^{(2)}_{22}
     \end{array}
     & U_1^{(1)} \big( V_2^{(1)} \big)^T \\ \hline
     U_2^{(1)} \big( V_1^{(1)} \big)^T & 
     \small
      \begin{array}{c|c}
   M^{(2)}_{33} & U_3^{(2)} \big( V_4^{(2)} \big)^T  \\ \hline
   U_4^{(2)} \big( V_3^{(2)} \big)^T & M^{(2)}_{44}
     \end{array}
          \end{array} \right].
\end{equation*}

The definition of a HODLR matrix of course depends on how the partitioning~\eqref{eq:toplevelpartitioning} is chosen on each level of the recursion. This choice is completely determined by the integer partitions
\begin{equation} \label{eq:integerpart}
 n = n_1 + n_2 + \cdots n_{2^p}, \qquad m = m_1 + m_2 + \cdots + m_{2^p},
\end{equation}
corresponding to the sizes $n_j\times m_j$, $j = 1,\ldots, 2^p$, of the diagonal blocks $M_{11}^{(p)}, \ldots, M_{2^p,2^p}^{(p)}$ on the lowest level of the recursion. Given specific
integer partitions~\eqref{eq:integerpart}, we denote the set of HODLR matrices of rank $k$ by $\mathcal{H}_{n\times m}(k)$.

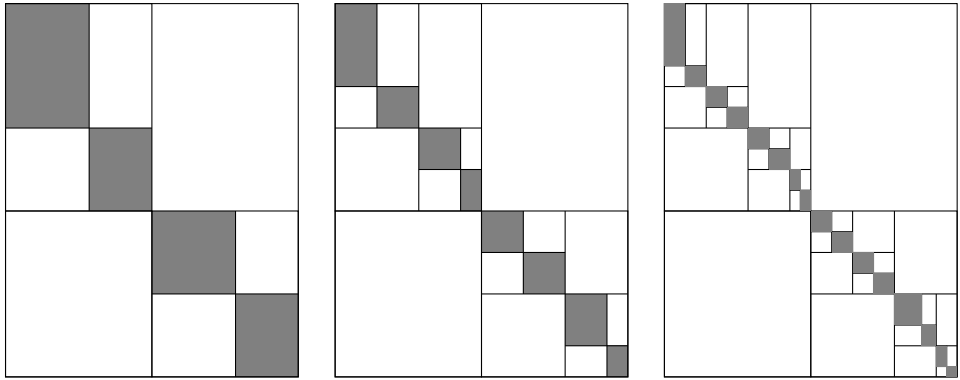
\begin{figure}[ht!]
\centering
\begin{tikzpicture}[scale=0.55]
\fill[gray] (2,6)--(3.5,6)--(3.5,4)--(2,4);
\fill[gray] (0,9)--(0,6)--(2,6)--(2,9);

\fill[gray] (3.5,4)--(5.5,4)--(5.5,2)--(3.5,2);
\fill[gray] (5.5,2)--(7,2)--(7,0)--(5.5,0);

\draw (0,0) rectangle (7,9);
\draw (0,0) rectangle  (3.5, 4);
\draw (3.5,0) rectangle  (7, 4);
\draw (0,4) rectangle  (3.5, 9);

\draw (0,6)--(3.5,6);
\draw (2,4)--(2,9);

\draw (3.5,2)--(7,2);
\draw (5.5,0)--(5.5,4);

\end{tikzpicture} 
\quad 
\begin{tikzpicture}[scale=0.55]
\fill[gray] (0,9)--(1,9)--(1,7)--(0,7);
\fill[gray] (1,7)--(2,7)--(2,6)--(1,6);

\fill[gray] (2,6)--(3,6)--(3,5)--(2,5);
\fill[gray] (3,5)--(3.5,5)--(3.5,4)--(3,4);

\fill[gray] (3.5,4)--(4.5,4)--(4.5,3)--(3.5,3);
\fill[gray] (4.5,3)--(5.5,3)--(5.5,2)--(4.5,2);
\fill[gray] (5.5,0.75)--(6.5,0.75)--(6.5,2)--(5.5,2);
\fill[gray] (6.5,0.75)--(7,0.75)--(7,0)--(6.5,0);

\draw (0,0) rectangle (7,9);
\draw (0,0) rectangle  (3.5, 4);
\draw (3.5,0) rectangle  (7, 4);
\draw (0,4) rectangle  (3.5, 9);

\draw (0,6)--(3.5,6);
\draw (2,4)--(2,9);

\draw (3.5,2)--(7,2);
\draw (5.5,0)--(5.5,4);
 
\draw (0,7)--(2,7);
\draw (1,6)--(1,9);

\draw (3,4)--(3,6);
\draw (2,5)--(3.5,5);

\draw (3.5,3)--(5.5,3);
\draw (4.5,2)--(4.5,4);
\draw (5.5,0.75)--(7,0.75);
\draw (6.5,0)--(6.5,2);

\end{tikzpicture}
\quad
\begin{tikzpicture}[scale=0.55]

\draw (0,0) rectangle (7,9);
\draw (0,0) rectangle  (3.5, 4);
\draw (3.5,0) rectangle  (7, 4);
\draw (0,4) rectangle  (3.5, 9);

\draw (0,6)--(3.5,6);
\draw (2,4)--(2,9);

\draw (3.5,2)--(7,2);
\draw (5.5,0)--(5.5,4);
 
\draw (0,7)--(2,7);
\draw (1,6)--(1,9);

\draw (3,4)--(3,6);
\draw (2,5)--(3.5,5);

\draw (3.5,3)--(5.5,3);
\draw (4.5,2)--(4.5,4);
\draw (5.5,0.75)--(7,0.75);
\draw (6.5,0)--(6.5,2);

\draw (0,7.5)--(1,7.5);
\draw (0.5,7)--(0.5,9);
\draw (1,6.5)--(2,6.5);
\draw (1.5,6)--(1.5,7);
\draw (2,5.5)--(3,5.5);
\draw(2.5,5)--(2.5,6);
\draw (3,4.5)--(3.5,4.5);
\draw (3.25,4)--(3.25,5);
\draw (3.5,3.5)--(4.5,3.5);
\draw (4,3)--(4,4);
\draw (4.5,2.5)--(5.5,2.5);
\draw (5,2)--(5,3);
\draw (6.75,0)--(6.75,0.75);
\draw (6.5,0.25)--(7,0.25);
\draw (5.5,1.25)--(6.5,1.25);
\draw (6.15,0.75)--(6.15,2);

\fill[gray] (0.5,7)--(1,7)--(1,7.5)--(0.5,7.5);
\fill[gray] (0,9)--(0,7.5)--(0.5,7.5)--(0.5,9);

\fill[gray] (1.5,6)--(2,6)--(2,6.5)--(1.5,6.5);
\fill[gray] (1,7)--(1,6.5)--(1.5,6.5)--(1.5,7);
\fill[gray] (2.5,5)--(3,5)--(3,5.5)--(2.5,5.5);
\fill[gray] (2,6)--(2,5.5)--(2.5,5.5)--(2.5,6);
 \fill[gray] (3.25,4)--(3.5,4)--(3.5,4.5)--(3.25,4.5);
 \fill[gray] (3,5)--(3,4.5)--(3.25,4.5)--(3.25,5);
 \fill[gray] (4,3)--(4.5,3)--(4.5,3.5)--(4,3.5);
 \fill[gray] (3.5,4)--(3.5,3.5)--(4,3.5)--(4,4);
 \fill[gray] (5,2)--(5.5,2)--(5.5,2.5)--(5,2.5);
 \fill[gray] (4.5,3)--(4.5,2.5)--(5,2.5)--(5,3);
 \fill[gray] (6.15,1.25)--(6.5,1.25)--(6.5,0.75)--(6.15,0.75);
\fill[gray] (5.5,2)--(5.5,1.25)--(6.15,1.25)--(6.15,2);
 \fill[gray] (6.75,0)--(7,0)--(7,0.25)--(6.75,0.25);
 
 \fill[gray] (6.5,0.75)--(6.5,0.25)--(6.75,0.25)--(6.75,0.75);

\end{tikzpicture}
\caption{Illustration of HODLR matrices for $p=2$, $p=3$, and $p=4$. The diagonal blocks (grey) are stored as dense matrices, while 
the off-diagonal blocks (white) are stored in terms of their low-rank factors. }
\label{fig:h_hodlr_matrices}
\end{figure}

\subsection{Operations in the HODLR format}
\label{sec:hodlr_arithmetics}

Assuming that the integer partitions~\eqref{eq:integerpart} are balanced, $p =\calO(\log \tilde n)$ with $\tilde n = \max\{n,m\}$, and $k = \calO(1)$, the storage of $M \in \mathcal H_{n\times m}(k)$ in 
the HODLR format requires $\mathcal O( \tilde n \log \tilde n)$ memory. Various matrix operations with HOLDR matrices can be preformed with linear-polylogarithmic complexity. Table~\ref{table:complexity_HODLR} 
summarizes the operations needed in this work; we refer to, e.g.,~\cite{Ballani2016,Hackbusch2015} for more details. In order to perform operations including two HODLR matrices, 
the corresponding partitions ought to be compatible.

The operations listed in Table~\ref{table:complexity_HODLR} with subscript $\h$ employ recompression in order 
to limit the increase of off-diagonal ranks. In this paper recompression is done adaptively, such that the $2$-norm approximation error in each off-diagonal block is bounded by a prescribed truncation tolerance $\epsilon$.  For matrix addition, recompression is done only after adding two off-diagonal blocks, whereas multiplying HOLDR matrices and computing the Cholesky decomposition requires recompression in intermediate steps.   
    
\begin{table}[ht!]
 \caption{Complexity of operations involving HODLR matrices: $M\in \mathcal H_{n\times n}(k)$ symmetric positive definite,  $T\in \h_{m\times m}(k)$ triangular and 
 invertible,  $M_1,M_2 \in \mathcal H_{n\times m}(k)$, $M_3\in \h_{m\times p}(k)$,  $B\in \R^{m\times p}$,  $v\in \R^m$.}
\label{table:complexity_HODLR}
\centering \begin{tabular}{rcl}
 \hline
 Operation &  &Computational complexity \\ \hline
 Matrix-vector multiplication:   $M_1 v$ &  & $\mathcal{O}(k\tilde{n}\log \tilde{n})$, with $\tilde{n} = \max\lbrace n,m \rbrace$ \\
 Matrix addition: $M_1 +_{\mathcal{H}} M_2$  &  & $\mathcal{O}(k^2\tilde{n}\log \tilde{n})$, with $\tilde{n} = \max\lbrace n,m \rbrace$  \\
 Matrix-matrix multiplication: $M_2 *_{\mathcal{H}} M_3$ &  & $\mathcal{O}(k^2\tilde{m}\log^2 \tilde{m})$, with $\tilde{m} = \max\lbrace n,m,p \rbrace$  \\
 Cholesky decomposition: $\h\operatorname{-Cholesky}(M)$ & & $\mathcal{O}(k^2 n\log^2 n)$  \\ 
 Solution of  triangular system: $T^{-1} B$ & &$\mathcal{O}(k m\log m)$  \\ 
 Multiplication with $(\text{triangular})^{-1}$: $M_1 *_{\mathcal{H}} T^{-1}$  & & $\mathcal{O}(k^2\tilde{n}\log^2 \tilde{n})$, with $\tilde{n} = \max\lbrace n,m \rbrace$ \\ \hline
 \end{tabular}
 \end{table}
 
In this work we also need to extract submatrices of HODLR matrices. Let $M\in \mathcal H_{n\times n}(k)$, associated with an 
integer partition $n = n_1 + \cdots + n_{2^p}$, and consider a subset of indices $C \subset \{1,\ldots,n\}$. Then the submatrix $M(C,C)$ is again a HODLR matrix. To see this, consider the 
partitioning~\eqref{eq:toplevelpartitioning} and let $C = C_1 \cup C_2$ with $C_1 = C \cap [1, n_1^{(1)}]$ and $C_2 = C \cap [n_1^{(1)}+1, n]$, where $n^{(1)}_1$ is the size of $M^{(1)}_{11}$.  
Then
\begin{eqnarray*}
  M(C,C) &=&  \left[ \begin{array}{c|c}
   M^{(1)}_{11}(C_1,C_1) & M^{(1)}_{12}(C_1,C_2) \\ \hline
   M^{(1)}_{21}(C_2,C_1) & M^{(1)}_{22}(C_2,C_2) \\ 
     \end{array} \right] \\ &=& \left[ \begin{array}{c|c}
   M^{(1)}_{11}(C_1,C_1) & U_1^{(2)}(C_1,:) \big( V_2^{(2)}(C_2,:) \big)^T \\ \hline
   U_2^{(2)}(C_2,:) \big( V_1^{(2)}(C_1,:) \big)^T & M^{(1)}_{22}(C_2,C_2) \\ 
     \end{array} \right].
\end{eqnarray*}
Hence, the off-diagonal blocks again have rank at most $k$. Applying this argument recursively to $M^{(1)}_{11}(C_1,C_1)$, $M^{(1)}_{22}(C_2,C_2)$ establishes $M(C,C) \in \mathcal H_{m\times m}(k)$, associated with the integer partition
\[
  |C| =: m = m_1 + m_2 + \cdots m_{2^p}, 
\]
where $m_1$ is the cardinality of $C \cap [1,n_1]$, $m_2$ is the cardinality of $C \cap [n_1+1,n_1+n_2]$, and so on. Note that it may happen that some $m_j = 0$, in which case the corresponding blocks in the HODLR format vanish. Formally, this poses no problem in the definition and operations with HOLDR matrices. In practice, these blocks are removed to reduce overhead.

\subsection{Computation of spectral projectors in the HODLR format}

The method presented in~\cite{KressnerSus2017} for computing spectral projectors of banded matrices is based on 
the dynamically weighted Halley iteration from~\cite{NakaBaiGygi2010, NakaHigh2013} for computing the matrix sign function. In this work, we also need a slight variation of that method for dealing with HODLR matrices.

Given a symmetric non-singular matrix $A$, the method from~\cite{NakaBaiGygi2010} uses an iteration
\begin{equation}
\label{eq:dwh_iteration}
X_{k+1} =  \frac{b_k}{c_k}X_k  + \Big( a_k - \frac{b_k}{c_k}\Big)X_k(I+c_k X^T_k X_k)^{-1}, \quad X_0 = A/\alpha, 
\end{equation}
that converges globally cubically to $\sign(A)$. The parameter $\alpha > 0 $ is such that $\alpha\gtrsim \Vert A \Vert_2$. The parameters $a_k, b_k, c_k$ are computed by
\begin{equation}
\label{eq:qdwh_parameters_abc}
 a_k = h(l_k),\quad b_k = (a_k-1)^2/4, \quad c_k = a_k + b_k - 1,
\end{equation}
where $l_k$  is determined by the recurrence 
\begin{equation}
\label{eq:qdwh_parameter_l}
l_k = l_{k-1}(a_{k-1} + b_{k-1}l^{2}_{k-1})/(1 + c_{k-1}l^{2}_{k-1}), \quad k \geq 1,
\end{equation}
with a lower bound $l_{0}$ for $\sigma_{\min}(X_0)$, and the function $h$ is given by
\begin{equation*}
h(l) = \sqrt{1+ \gamma} + \frac{1}{2}\sqrt{8 - 4\gamma + \frac{8(2 - l^2)}{l^2\sqrt{1 + \gamma}}},\quad \gamma = \sqrt[3]{\frac{4(1 - l^2)}{l^4}}.
\end{equation*}
In summary, except for $\alpha$ and $l_0$ the parameters determining~\eqref{eq:dwh_iteration} are simple and cheap to compute.

The algorithm {\tt hQDWH} presented in~\cite{KressnerSus2017} for banded $A$ is essentially an implementation of~\eqref{eq:dwh_iteration} in the HODLR matrix arithmetic, with 
one major difference. Following~\cite{NakaHigh2013}, the first iteration of {\tt hQDWH} avoids the computation of the Cholesky factorization for the evaluation of $(I+c_0 X^T_0 X_0)^{-1} = (I+c_0/\alpha^2 A^2)^{-1}$ in 
the first iteration. Instead, a QR decomposition of a $2n\times n$ matrix $\begin{bmatrix}\sqrt{c_0} X_0 \\ I\end{bmatrix} $ is computed. This improves numerical 
stability and allows us to safely determine spectral projectors even for relatives gaps of order $10^{-16}$. For reasons explained 
in~\cite[Remark 3.1]{KressnerSus2017}, existing algorithms for performing QR decompositions of HODLR matrices come with various drawbacks. When $A$ is a HODLR matrix, 
Algorithm~\ref{alg:hdwh} therefore uses a Cholesky decomposition (instead of a QR decomposition) in the first step as well. In turn, as will be 
demonstrated by numerical experiments in section~\ref{sec:experiments}, this restricts the application 
of the algorithm to matrices with relative spectral gaps of order  $10^{-8}$ or larger. We do not see this as a major disadvantage in the setting under consideration. The relative gap is controlled by the choice of the shift $\mu$ in our D\&C method and tiny relative spectral gaps can be easily avoided by adjusting $\mu$.

The inexpensive estimation of $\alpha,l_0$ for banded $A$ is discussed in~\cite{KressnerSus2017}. For a HODLR matrix $A$, we determine $\alpha$ and $l_0$ by applying a few steps of the (inverse) power method to $A^2$ and $A^2/\alpha^2$, respectively.

\begin{algorithm}[h!]
    \caption{\text{hDWH algorithm}} 
    \label{alg:hdwh}
   \textbf{Input}: Symmetric HODLR matrix $A\in \R^{n\times n}$, truncation tolerance $\epsilon>0$, stopping tolerance $\varepsilon>0$. \\
   \textbf{Output}: Approximate spectral projectors $\Pi_{<0}$ and $\Pi_{>0}$ in the HODLR format.
    
   \begin{algorithmic}[1]
    
   \If{$A$ is banded}  
    \State Compute $\Pi_{<0}$ and $\Pi_{>0}$ using the {\tt hQDWH} algorithm~\cite{KressnerSus2017}.
    \Else
    \State Compute initial parameters $\alpha\gtrsim \Vert A \Vert_2$ via power iteration on $A^2$ and $l_0 \lesssim \sigma_{\min}(A/\alpha)$ via inverse power iteration on
   $A^2/\alpha^2$.
   \State $X_0 = A/\alpha$. 
   \State $k = 0$.
   \While{$\vert 1 - l_k\vert > \varepsilon$} \label{alg_stopping}   
     \State Compute $a_k$, $b_k$, $c_k$ according to the recurrence~\eqref{eq:qdwh_parameters_abc}.
      \State $W_k = \h\operatorname{-Cholesky}(I + c_kX_k^{T}*_{\h}X_k)$. \label{alg_chol1} 
         \State $Y_k = X_k*_{\h}W_k^{-1}$. 
         \State $V_k = Y_k *_{\h} W_k^{-T}$. 
        \State  $X_{k+1} = \frac{b_k}{c_k}X_k +_{\h} \left(a_k - \frac{b_k}{c_k}\right)V_k$. \label{alg_chol}  
     \State $k = k+1$.
     \State Compute $l_k$ according to the recurrence~\eqref{eq:qdwh_parameter_l}.
   \EndWhile
   \State Set $U = X_k$. \label{alg_sign} 
   \State Return $\Pi_{<0} =\frac{1}{2}(I -U)$ and $\Pi_{>0} =\frac{1}{2}(I +U)$. \label{alg_sp} 
  \EndIf
  \end{algorithmic}
\end{algorithm}

Assuming that a constant number of iterations is needed and that the HOLDR ranks of all intermediate quantities are bounded by $k$, the complexity of Algorithm~\ref{alg:hdwh} is $\calO(k^2n\log^2 n)$.

\section{Computation of invariant subspace basis in the HODLR format}  
\label{sec:invariant_subspace}

This section addresses the efficient extraction of a basis for the range of a spectral projector $\Pi_{<\mu}$ given in the HODLR format.

Assuming that $\rank (\Pi_{<\mu}) = \nu$, the most straightforward approach to obtain a basis for $\range(\Pi_{<\mu})$ is to simply take its first $\nu$ columns. Numerically, 
this turns out to be a terrible idea, especially when $A$ is banded.

\begin{example} \label{example:badcond} \rm
Let $n$ be even and let $A\in \R^{n\times n}$ be a symmetric banded matrix with bandwidth $b$ and eigenvalues distributed uniformly in $[-1, \hskip 3pt -10^{-1}]\cup [10^{-1}, \hskip 3pt 1]$. In 
particular, $\rank(\Pi_{<0}) = n/2$. Figure~\ref{fig:condition_first_half} shows that the condition number of the first $n/2$ columns of $\Pi_{<0}$ grows dramatically as $n$ increases. By computing a 
QR decomposition of these columns, we obtain an orthonormal basis $Q_{1} \in \R^{n\times n/2}$. This basis has perfect condition number but, as Table~\ref{table:condition_first_half} shows, it represents a r
ather poor approximation of $\range(\Pi_{<0})$.
\begin{figure}[tbhp]
\begin{center}
\includegraphics[width=0.4\textwidth]{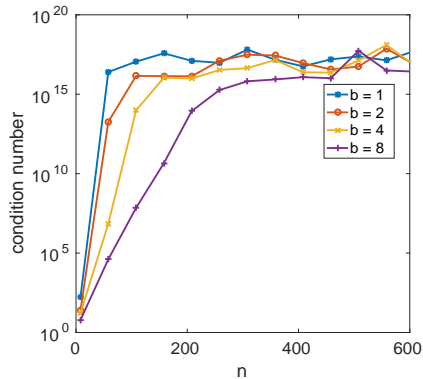}
\end{center}
\caption{Condition number of the first $n/2$ columns of the spectral projector $\Pi_{<0}$ for the matrix described in Example~\ref{example:badcond} with bandwidths $b = 1,2,4,8$.}
\label{fig:condition_first_half}
\end{figure}

\begin{table}[h!]
\caption{Angles (in radians) between $\range(\Pi_{<0})$ and $\range(Q_{1})$, with
$Q_{1}$ an orthonormal basis for $\range(\Pi_{<0}(:, 1:\frac{n}{2}))$.}
 \label{table:condition_first_half}
\centering
\begin{tabular}[t]{c||c}
$n$  &$\measuredangle(\range(\Pi_{<0}),\range(Q_{1})))$\\
 \hline
 \hline
 $64$  &$4.4916e-02$\\
 $256$  &$1.5692e+00$\\
 $1024$   &$1.5700e+00$ \\
 $4096$ &$1.5707e+00$\\
 \hline
\end{tabular}
\end{table}
\end{example}

There exist a number of approaches that potentially avoid the problems observed in Example~\ref{example:badcond}, such as applying a QR factorization 
with pivoting~\cite[Chapter 5.4]{GolVanL2013} to $\Pi_{<0}$. None of these approaches has been realized in the HODLR format. In fact, techniques like pivoting across blocks appear to be incompatible with the format.

In the following, we develop a new algorithm for computing a basis for $\range(\Pi_{<\mu})$ in the HODLR format, which consists of two steps: 
(1) We first determine a set of well-conditioned columns of $\Pi_{<\mu}$ by performing a Cholesky factorization with \emph{local} pivoting. As we will see below, the number of obtained 
columns is generally less but not much less than $\nu$. (2) A randomized algorithm is applied to complete the columns to a basis of $\range(\Pi_{<\mu})$.

\subsection{Column selection by block Cholesky with local pivoting}

The spectral projector $\Pi_{<\mu}$ is not only symmetric positive semidefinite but it is also idempotent. The (pivoted) Cholesky factorization of such matrices has particular properties.

\begin{theorem}[{\cite[Theorem 10.9]{Higham1996}}]
\label{thm:cholp}
 Let $B\in \R^{n\times n}$ be a symmetric positive semidefinite matrix of rank $r$. Then there is a permutation matrix $P$ such that $P^TBP$ admits a Cholesky factorization:
 \begin{equation*}
  P^TBP = R^T R, \quad R = \begin{bmatrix} R_{1} &R_{2}\\ 0 &0\end{bmatrix},
 \end{equation*}
where $R_{1}$ is a $r\times r$ upper triangular matrix with positive diagonal elements.
\end{theorem}

Note that, by the invertibility of $R_{1}$, the first $r$ columns of $BP$ as well as $[R_{1}\quad R_{2}]^T$ form a basis for $\range(B)$. The latter turns out to be orthonormal if $B$ is idempotent.

\begin{lemma}[{\cite[Corollary 1.2.]{MolerStew1978}}] 
 \label{thm:on_columns}
 Suppose, in addition to the hypotheses of Theorem~\ref{thm:cholp}, that $B^2 = B$. Then 
 $$[R_{1}\quad R_{2}]\begin{bmatrix}R^{T}_{1}\\ R^{T}_{2}\end{bmatrix} = I_r.$$ 
\end{lemma}

The algorithm described in~\cite[Chapter 10]{Higham1996} for realizing Theorem~\ref{thm:cholp} chooses the maximal diagonal element as the pivot in every step of the standard 
Cholesky factorization algorithm. In turn, the diagonal elements of the Cholesky factor are monotonically decreasing and it is safe to decide which ones are considered zero numerically. Unfortunately, this algorithm, 
which will be denoted by {\tt cholp} in the following, cannot be applied to $\Pi_{<\mu}$ because the diagonal pivoting strategy destroys the HODLR format.
Instead, we use {\tt cholp} only for the (dense) diagonal blocks  of $\Pi_{<\mu}$. 

To illustrate the idea of our algorithm, we first consider a general symmetric positive semidefinite HODLR matrix $M$ of level $1$, which takes the form
\begin{equation*} \label{eq:HODLR1}
M = \begin{pmatrix}
               M_{11} & U_1 V_2^T\\
               V_2 U_1^T  &M_{22}\\ 
              \end{pmatrix}
\end{equation*}  
with dense diagonal blocks $M_{11}$, $M_{22}$. Applying {\tt cholp} to $M_{11}$ gives a decomposition $P_1^T M_{11} P_1 = R_{11}^T R_{11}$, with the diagonal elements of $R_{11}$ decreasing monotonically. As $M$, and 
in turn also $M_{11}$, will be chosen as a principal submatrix of $\Pi_{<\mu}$, Lemma~\ref{thm:on_columns} implies that $\|R_{11}\|_2 \le 1$. In particular, the diagonal elements of $R_{11}$ are bounded by $1$. Let $s_1$ 
denote the number of diagonal elements not smaller than a prescribed threshold $\delta$. As will be 
shown in Lemma~\ref{lemma:invR} below, choosing $\delta$ sufficiently close to $1$ ensures that $R_{11}(1:s_1,1:s_1)$ is well-conditioned. Letting $\pi_1$ denote the permutation associated
with $P_1$ and setting $C_1 = \pi_1(1:s_1)$, we have
\[
 M_{11}(C_1,C_1) = R_{11}(1:s,1:s)^T R_{11}(1:s,1:s).
\]
The Schur complement of this matrix in $M$ (without considering the rows and columns neglected in the first block) is given by
\begin{equation} \label{eq:sc}
 S = M_{22} - V_2 U_1(C_1,:)^T M_{11}(C_1,C_1)^{-1} U_1(C_1,:)V_2^{T} = M_{22} - \tilde R_{12}^T \tilde R_{12},
\end{equation}
where the rank of $\tilde R_{12} := R_{11}(1:s_1,1:s_1)^{-T} U_1(C_1,:) V_2^{T}$ is not larger than the rank of $U_1 V_2^T$.
We again apply {\tt cholp} to $S$ and only retain diagonal elements of the Cholesky factor $R_{22}$ larger or equal than $\delta$. Letting $C_2$ denote the corresponding indices and setting $s_2 = |C_2|$,
$C= C_1 \cup (n_1+C_2)$, where $n_1$ is the size of $M_{11}$,  we obtain the factorization
\[
 M(C,C) = \tilde R^T \tilde R \quad \text{with} \quad \begin{bmatrix}
           R_{11}(1:s_1,1:s_1) & \tilde R_{12}(:,C_2) \\
           0 & R_{22}(1:s_2,1:s_2)
          \end{bmatrix}.
\]
For a general HODLR matrix, we proceed recursively in an analogous fashion, with the difference that we now form submatrices of HODLR matrices (see section~\ref{sec:hodlr_arithmetics}) and the 
operations in~\eqref{eq:sc} are executed in the HODLR arithmetic.

The procedure described above leads to Algorithm~\ref{alg:cholp_col_sel}. Based on the complexity of operations stated in Table~\ref{table:complexity_HODLR}, the cost of the 
algorithm applied to an $n\times n$ spectral projector $\Pi_{<\mu} \in \h_{n\times n}(k)$ is $\calO(k^2n\log^2 n)$. In Line~\ref{line:sc} we update a HODLR matrix with a matrix given by its low-rank representation. This operation essentially corresponds to 
the addition of two HODLR matrices. Recompression is performed when computing all off-diagonal blocks, while dense diagonal blocks are updated by 
a dense matrix of the corresponding size. In Line~\ref{line:sc} we also enforce symmetry in the Schur complement.  

\begin{algorithm}[ht!]
    \caption{Incomplete Cholesky factorization with local pivoting for HODLR matrices}
    \label{alg:cholp_col_sel}
    \textbf{Input:} Positive semidefinite HODLR matrix $M\in \h_{n\times n}(k)$ of level $p$, tolerance $\delta > 0$.  \\
    \textbf{Output:} Indices $C \subset [1,n]$ and upper triangular HODLR matrix $\tilde R$ such that $M(C,C) \approx \tilde R^T\tilde R$, with $\tilde r_{ii} \geq \delta$ for $i = 1,\ldots,|C|$. 
   \begin{algorithmic}[1]
    \Function  {$[C,\tilde R] =$ \tt hcholp$\_${\tt inc}}{$M$}
   \If{$p = 0$}
    \State Compute $[R, \pi] = {\tt cholp}(M)$ such that $M(\pi,\pi) = R^T R$.
    \State Set $s$ such that $r_{11}\ge \delta, \ldots, r_{ss}\ge \delta$ and $r_{s+1,s+1} < \delta$ (or $s = n$).
    \State Return $C = \pi(1:s)$ and $\tilde R = R(1:s,1:s)$.
   \Else
    \State Partition $M = \begin{pmatrix}
               M_{11} & U_1 V_2^T\\
               V_2 U_1^T  &M_{22}\\ 
              \end{pmatrix}$.
    \State Call recursively $[C_1,\tilde R_{11}] = \text{\tt hcholp$\_${\tt inc}}(M_{11})$.
    \State Compute $\tilde U_1 = \tilde R_{11}^{-T} U_1(C_1,:)$. \label{line:u1}
    \State Compute $S = M_{22} -_{\h} V_2 \tilde U_1^T \tilde U_1 V_2^T$. \label{line:sc}
    \State Call recursively $[C_2,\tilde R_{22}] = \text{\tt hcholp$\_${\tt inc}}(S)$.

    \State Return $C = C_1 \cup (n_1 + C_2)$ and HODLR matrix $\tilde R = \begin{bmatrix}
                                                              \tilde R_{11} & \tilde U_1 V_2(C_2,:)^T \\
                                                              0 & \tilde R_{22}
                                                             \end{bmatrix}$.
   \EndIf
   \EndFunction
  \end{algorithmic}
\end{algorithm}

\subsubsection{Analysis of Algorithm~\ref{alg:cholp_col_sel}}

The indices $C$ selected by Algorithm~\ref{alg:cholp_col_sel} applied to $\Pi_{<\mu}$ need to attain two goals: (1) $\Pi_{<\mu}(:,C)$ has moderate condition number, (2) $|C|$ is not much smaller 
than the rank of $\Pi_{<\mu}$. In the following analysis, we show that the first goal is met when choosing $\delta$ sufficiently close to $1$. The attainment of the second goal is 
demonstrated by the numerical experiments in section~\ref{sec:experiments}.

Our analysis needs to take into account that Algorithm~\ref{alg:cholp_col_sel} is affected by error due to truncation in the HODLR arithmetic. On the one hand, 
the input matrix, the spectral projector $\Pi_{<\mu}$ computed by Algorithm~\ref{alg:hdwh}, is not exactly idempotent:
\begin{equation}
\label{eq:error_idempotent}
\Pi^2_{<\mu} = \Pi_{<\mu} + F, 
\end{equation}
with a symmetric perturbation matrix $F$ of small norm. On the other hand, the incomplete Cholesky factor $\tilde R$ returned by 
Algorithm~\ref{alg:cholp_col_sel} is inexact as well:
\begin{equation} 
\label{eq:error_chol}
\Pi_{<\mu}(C, C) = \tilde R^T \tilde R + E,   
\end{equation}
with another symmetric perturbation matrix $E$ of small norm.
For a symmetric matrix $\Pi_{<\mu}$ satisfying~\eqref{eq:error_idempotent}, Theorem 2.1 in~\cite{MolerStew1978} shows that 
\begin{equation} 
 \Vert \Pi_{<\mu} \Vert_2 \leq 1+ \Vert F\Vert_2.  \label{eq:bound_norm_sp} 
\end{equation}
The following lemma establishes a bound on the norm of the inverse of $\Pi_{<\mu}(C, C)$.
\begin{lemma} \label{lemma:invR}
 With the notation introduced above, set $\varepsilon_{\h} = \|E\|_2 + \|F\|_2$ and $r = \vert C\vert$, and suppose that 
$1-\delta^2 + \varepsilon_{\h}< 1/r$. Then $\|\Pi_{<\mu}(C, C)^{-1}\|_2 \le \frac{1}{r}\frac{1}{\delta^2 - 1 + 1/r - \varepsilon_{\h}}$.
\end{lemma}
\begin{proof}
 Using~\eqref{eq:error_chol} and~\eqref{eq:bound_norm_sp}, we obtain
 \[
  \|\tilde R^T\tilde R\|_2 \le \|\Pi_{<\mu}(C, C)\|_2 + \|E\|_2 \le \|\Pi_{<\mu}\|_2 + \|E\|_2 \le 1 + \varepsilon_{\h}.
 \]
  We now decompose $\tilde R = D+T$, such that $D$ is diagonal with $d_{ii}  = \tilde r_{ii} \ge \delta$ and $T$ is strictly upper triangular. Then
 \[
  \|D^2 + T^TD+DT + T^TT \|_2 \le 1 + \varepsilon_{\h}.
 \]
 Because the matrix on the left is symmetric, this implies
 \[
  \lambda_{\max}( D^2 + T^TD +DT + T^TT) \le 1 + \varepsilon_{\h} \ \Rightarrow\ \lambda_{\max}( T^TD +DT + T^TT) \le 1 -\delta^2 + \varepsilon_{\h}.
 \]
 On the other hand,
 \[
  \lambda_{\min}( T^TD +DT + T^TT) \ge \lambda_{\min}( T^TD +DT) \ge -(r-1) \lambda_{\max} ( T^TD +DT ) \ge - (r-1)(1 -\delta^2 + \varepsilon_{\h}),
 \]
 where the second inequality uses that the trace of $T^TD +DT$ is zero and hence its eigenvalues sum up to zero. In summary,
 \[
  \|T^TD +DT + T^TT\|_2 \le (r-1)(1 -\delta^2 + \varepsilon_{\h}),
 \]
 and $\Pi_{<\mu}(C, C) = D^2 + \tilde E$ with $\|\tilde E\|_2 \le (r-1)(1 -\delta^2) + r\varepsilon_{\h}$. This completes the proof because
 \begin{eqnarray*}
  \|\Pi_{<\mu}(C, C)^{-1}\|_2 &\le& \|D^{-2}\|_2 \|( I + D^{-2} \tilde E)^{-1}\|_2 \le \frac{1}{\delta^2 - (r-1)(1 -\delta^2) - r\varepsilon_{\h}} \\
  &=& \frac{1}{r}\frac{1}{\delta^2 - 1 + 1/r - \varepsilon_{\h}},
  \end{eqnarray*}
 where the inverse exists under the conditions of the lemma.
\end{proof}
The following theorem shows that the columns $\Pi_{<\mu}(:, C)$ selected by Algorithm~\ref{alg:cholp_col_sel} have an excellent condition number if $\delta$ is sufficiently close to one and the 
perturbations introduced by the HODLR arithmetic remain small.
\begin{theorem}
\label{thm:cond_chol}
Let $C$ denote the set of $r$ indices returned by Algorithm~\ref{alg:cholp_col_sel} and suppose that the conditions~\eqref{eq:error_idempotent} and~\eqref{eq:error_chol} as well as the condition of Lemma~\ref{lemma:invR} are satisfied. Then it holds for the $2$-norm condition number of $\Pi_{<\mu}(:, C)$ that
 \begin{equation*}
  \kappa(\Pi_{<\mu}(:, C))  \leq \frac1r \frac{1 + \varepsilon_{\h}}{\delta^2 - 1 + 1/r - \varepsilon_{\h}} = 1+2r(1-\delta) + \mathcal O((1-\delta)^2 + \varepsilon_{\h}).
 \end{equation*}
\end{theorem}
\begin{proof}
By definition, $\kappa(\Pi_{<\mu}(:, C)) = \Vert \Pi_{<\mu}(:, C)\Vert_2 \Vert \Pi_{<\mu}(:, C)^{\dagger}\Vert_2$. From~\eqref{eq:bound_norm_sp} we get \begin{equation}
\label{eq:column_norm}
\Vert \Pi_{<\mu}(:, C)\Vert_2 \leq \Vert \Pi_{\mu} \Vert_2 \leq 1+ \Vert F\Vert_2. 
\end{equation}
To bound the second factor, we note that 
$
  \Vert \Pi_{<\mu}(:, C)^{\dagger} \Vert_2 \le \Vert \Pi_{<\mu}(C, C)^{-1} \Vert_2$ and apply Lemma~\ref{lemma:invR}. Using the two bounds, 
the statement follows.  
\end{proof}
The condition of Lemma~\ref{lemma:invR}, $1-\delta^2 \lesssim 1/r$, requires $\delta$ to be very close to $1$. We conjecture that this condition can be 
improved to a distance that is proportional to $1/\log_2 r$ or even a constant independent of $r$. The latter is what we observe in the numerical experiments; choosing $\delta$ constant and letting $r$ grow does 
not lead to a deterioration of the condition number.

\subsection{Range correction}

As earlier, let $C$ denote a set of indices obtained by Algorithm~\ref{alg:cholp_col_sel} for a threshold $\delta$, and $r = \vert C\vert$. We recall that 
the dimension of the column space of $\Pi_{<\mu}$ can be easily computed knowing that $\trace(\Pi_{<\mu}) = \rank(\Pi_{<\mu}) = \nu$.  
If  $r = \nu$, then it only remains to perform the orthogonalization to get an orthonormal basis of $\range(\Pi_{<\mu})$. However, depending on the choice of $\delta$, in practice it 
can occur that the cardinality of $C$ is smaller 
than $\nu$, which implies that the selected columns cannot span the column space of $\Pi_{<\mu}$. In this case additional vectors need to be computed to get a complete orthonormal basis for $\range(\Pi_{<\mu})$. 
 
An orthonormal basis for $\range(\Pi_{<\mu}(:, C))$ in the HODLR format can be computed using a method suggested in~\cite{Lintner2002}, and it is given as 
\begin{equation}
\label{eq:onb_delta}
 \Pi_{<\mu}(:, C)*_{\h} \tilde{R}^{-1}.
\end{equation}

The biggest disadvantage of the method in~\cite{Lintner2002} is the loss the orthogonality in badly conditioned problems, caused by squaring of the condition number when computing $\tilde{R}$. However, 
choosing only well-conditioned subset of columns of $\Pi_{<\mu}$ allows us to avoid dealing with badly conditioned problems, and thus prevents potential loss of orthogonality in~\eqref{eq:onb_delta}.  

In case $r < \nu$, we complete the basis~\eqref{eq:onb_delta} to an orthonormal basis for $\range(\Pi_{<\mu})$, by computing
an orthonormal basis of the orthogonal complement of $\range(\Pi_{<\mu}(:, C))$ in  $\range(\Pi_{<\mu})$. First we detect the orthogonal complement of $\range(\Pi_{<\mu}(:, C))$ in  
$\range(\Pi_{<\mu})$.

\begin{lemma}
 \label{thm:orth_complement_cdelta}
If $(\range (\Pi_{<\mu}(:, C)))^{\bot}$ is the orthogonal complement of $\range (\Pi_{<\mu}(:, C))$, then 
\begin{equation*}
\label{eq:complement_delta}
R^{\bot}_{\Pi_{<\mu}, C} : = (\range (\Pi_{<\mu}(:, C)))^{\bot}\cap \range(\Pi_{<\mu})
\end{equation*}
is the orthogonal complement of $\range (\Pi_{<\mu}(:, C))$ in the vector space $\range(\Pi_{<\mu})$. Moreover, $\dim(R^{\bot}_{\Pi_{<\mu}, C}) = \rank(\Pi_{<\mu}) - r$. 
\end{lemma}

\begin{proof}

The statements follow directly from the definition of $R^{\bot}_{\Pi_{<\mu}, C}$.
\end{proof}

Using~\eqref{eq:onb_delta} we construct an orthogonal projector
\begin{equation}
\label{eq:orth_projection_compl}
P_{C^{\bot}} =  I - \Pi_{<\mu}(:, C)*_{\h} \tilde{R}^{-1}*_{\h}(\Pi_{<\mu}(:, C)*_{\h} \tilde{R}^{-1})^T
\end{equation}
onto  $(\range (\Pi_{<\mu}(:, C)))^{\bot}$. From~\eqref{eq:onb_delta} it steadily follows that $\range (P_{C^{\bot}} \Pi_{<\mu}) = R^{\bot}_{\Pi_{<\mu}, C}$.  Thus
computing an orthonormal basis for $P_{C^{\bot}} \Pi_{<\mu}$ will allow us to obtain a complete orthonormal basis for $\range(\Pi_{<\mu})$.  

To this end, we employ a randomized algorithm~\cite{HalkoMartinsson2011} 
to compute an orthonormal basis of dimension $ \nu - r$ for $\range(P_{C^{\bot}} \Pi_{<\mu})$: (1) we first 
multiply $P_{C^{\bot}} \Pi_{<\mu}$ with a random matrix $X\in \R^{n\times (\nu - r +p)}$, where $p$ is an oversampling parameter; (2) we compute its QR decomposition. As singular values of $\Pi_{<\mu}$ are either 
unity or zero, multiplication with the orthogonal projector $P_{C^{\bot}}$, generated by the linearly independent columns $C$, gives a matrix whose singular values are well-separated as well. In particular, the
resulting matrix has $\nu - r$ singular values equal to $1$, and the others equal to zero. Indeed, in exact arithmetics $P_{C^{\bot}} \Pi_{<\mu}$ has the exact 
rank $\nu - r$, and then oversampling is not required~\cite{HalkoMartinsson2011}. However, due to the formatted arithmetics, we use a small oversampling parameter $p$ to improve accuracy. As
we require only $ \nu -r$ columns to complete the basis for $\range(\Pi_{<\mu})$, finally we keep only the first $\nu - r$ columns of the orthonormal factor.

 A pseudo-code for computing a complete orthonormal basis for $\range(\Pi_{<\mu})$ is given in Algorithm~\ref{alg:basis_complete}. Note  
 that $\Pi_{<\mu}(:, C)$ is a rectangular HODLR matrix, obtained by extracting columns 
 with indices $C$ of a HODLR matrix, as explained in section~\ref{sec:hodlr}. This implies that 
 the complexity of operations stated in Table~\ref{table:complexity_HODLR} carries over for the operations involving HODLR matrices in Algorithm~\ref{alg:basis_complete}. The complexity of the algorithm also depends on 
 the number of the missing basis vectors. However, in our experiments we observe that $\nu - r$ is very small with respect to $\nu$ and $n$ for choice of $\delta$ we use, which makes the 
 cost of operations in Line~\ref{basis_comp_project} and Line~\ref{qr_correct} negligible. In the setup when $\nu\approx n/2$, the overall complexity of Algorithm~\ref{alg:basis_complete} is governed by solving
 a triangular system in Line~\ref{orth_updated} or Line~\ref{orth}, i.e. it is $\calO(k^2 n\log^2n)$. 

\begin{algorithm}[ht!]
    \caption{Computation of a complete orthonormal basis for $\range(\Pi_{<\mu})$}
    \label{alg:basis_complete}

    \textbf{Input:} Spectral projector $\Pi_{<\mu}\in \R^{n\times n}$ in the HODLR format with $\rank(\Pi_{<\mu}) = \nu$, column indices $C$ and the Cholesky factor $\tilde R$ returned by 
    Algorithm~\ref{alg:cholp_col_sel}, an oversampling 
     parameter $p$. \\
    \textbf{Output:} Orthonormal matrix $Q_{<\mu}\in \R^{n\times \nu}$ such that $\range(Q_{<\mu}) = \range(\Pi_{<\mu})$.  
  \begin{algorithmic}[1]
    \If{$\vert C\vert < \nu$}
    \State Generate a random matrix $X \in \R^{n\times (\nu - r +p)}$, for $r = \vert C\vert$.
    \State $ Z =  \Pi_{<\mu} X - \Pi_{<\mu}(:, C)(\tilde{R}^{-1}(\tilde{R}^{-T}(\Pi_{<\mu}(C,:)(\Pi_{<\mu}X))))$. \label{basis_comp_project}
    \State Compute $[Q_c, \sim,\sim] = {\tt qr}(Z,0)$. \label{qr_correct}
    \State Return $Q_{<\mu} = [\Pi_{<\mu}(:, C)*_{\h}\tilde{R}^{-1}\quad Q_c(:, 1:\nu - r)]$. \label{orth_updated} 
   \Else     
   \State Return $Q_{<\mu} = [\Pi_{<\mu}(:, C)*_{\h}\tilde{R}^{-1}]$. \label{orth}
   \EndIf
  \end{algorithmic}
\end{algorithm}

\subsubsection{Storing additional columns}
When range correction is performed, we additionally need to store tall-and-skinny matrix $Q_c$ from Algorithm~\ref{alg:basis_complete}. The idea is to incorporate columns of $Q_c$
into an existing HODLR matrix $\Pi_{<\mu}(:, C)*_{\h}\tilde{R}^{-1}$ of size $n\times r$ to get a HODLR matrix of size $n\times \nu$. More specifically, we append $\nu - r$ columns after 
the last column of $\Pi_{<\mu}(:, C)*_{\h}\tilde{R}^{-1}$, by enlarging all blocks of $\Pi_{<\mu}(:, C)*_{\h}\tilde{R}^{-1}$ that contain the last column. Recompression is performed when updating the off-diagonal 
blocks. It is expected that the off-diagonal ranks in the updated blocks grow, however, numerical experiments in section~\ref{sec:experiments} demonstrate that the increase is not significant.  

%

\section{Divide-and-conquer method in the HODLR format}
\label{sec:dc_hodlr}

In this section we give the overall spectral divide-and-conquer method for computing the eigenvalue decomposition of a symmetric banded matrix $A\in \R^{n\times n}$. 
For completeness, we also include a pseudocode given in Algorithm~\ref{alg:hsdc}. In the following we discuss several details related to
its implementation and provide the structure of the eigenvectors matrix. 

\begin{algorithm}[ht!]
    \caption{Spectral divide-and-conquer algorithm in the HODLR format ({\tt hSDC})}
    \label{alg:hsdc}
    \textbf{Input:} A symmetric banded or HODLR matrix $A\in \R^{n\times n}$.  
    
  \textbf{Output:} A structured matrix $Q$ containing the eigenvectors of $A$ and a diagonal matrix $\Lambda$ containing the eigenvalues of $A$.
  
  \begin{algorithmic}[1]
  \Function {$[Q,\Lambda] =  $\tt hsdc} {$A$}
   
   \If{$n \leq n_{\stp}$}
     \State $[Q, \Lambda] = {\tt eig}(A)$.  \label{heig}
    \Else

   \State Compute $\mu = {\tt median}(\diag (A))$.   \label{hshift}
   \State Compute $\Pi_{<\mu}$ and $\Pi_{>\mu}$ in the HODLR format by applying Algorithm~\ref{alg:hdwh} to $A-\mu I$. \label{hsp2}
 
   \State Compute column indices $C_{<\mu}$ and $C_{>\mu}$ by applying Algorithm~\ref{alg:cholp_col_sel} to $\Pi_{<\mu}$ and $\Pi_{>\mu}$. \label{hindices}
   \State Compute $Q_{<\mu}$ and $Q_{>\mu}$ by applying Algorithm~\ref{alg:basis_complete} to $\Pi_{<\mu}, C_{<\mu}$, and $\Pi_{>\mu}, C_{>\mu}$. \label{honb}
   \State Form $A_{<\mu} = Q_{<\mu}^{T}*_{\h}A*_{\h}Q_{<\mu}$ and $A_{>\mu} = Q_{>\mu}^{T}*_{\h}A*_{\h}Q_{>\mu}$. \label{hdc}
   \State Call recursively {$[Q_1, \Lambda_1]$ = \tt hsdc}$(A_{<\mu})$ and {$[Q_2, \Lambda_2]$  = \tt hsdc}$(A_{>\mu})$. 
   \State Set $Q \gets \begin{bmatrix} Q_{<\mu}  & Q_{>\mu} \end{bmatrix}*_{\h} \begin{bmatrix} Q_1 &0 \\ 0 &Q_2\end{bmatrix} $ and  $\Lambda = \begin{bmatrix} \Lambda_1 & 0 \\ 0 & \Lambda_2 \end{bmatrix}$.
   \EndIf  
   \EndFunction
 \end{algorithmic}
\end{algorithm}

\subsection{Computing the shift}
The purpose of computing shift $\mu$ is to split a problem of size $n$ into a two smaller subproblems of roughly the same size. In this work, the computation of 
a shift is performed by computing the median of $\diag(A)$, as proposed in~\cite{NakaHigh2013}. Although this way of estimating the median of eigenvalues may not be optimal, 
it is a cheap method and gives reasonably good results. For more details regarding the shift computation, we refer the reader to a discussion in~\cite{NakaHigh2013}. Moreover, we note that it remains an 
open problem to develop a better strategy for splitting the spectrum. 
 
\subsection{Terminating the recursion}
We stop the recursion when the matrix attains the minimal prescribed size $n_{\operatorname{stop}}$, and use \Matlab{} built-in function {\tt eig} to perform the final step of diagonalization.  
In a practical implementation, we set $n_{\operatorname{stop}}$ depending on the breakeven point of {\tt hQDWH} relative to {\tt eig} obtained in~\cite{KressnerSus2017}.  

\subsection{Matrix of eigenvectors}
For simplicity, without loss of generality we assume that for size of a given matrix $A$ holds $n = 2^s n_{\stp}$, for $s\in \N$. We say that 
Algorithm~\ref{alg:hsdc} performed level $l$ divide step, with $0\leq l < s$,  
if all matrices of size $n/2^l$ had been subdivided. 
 
The eigenvectors matrix is given as an implicit product of orthonormal HODLR matrices. After level $l$ divide step of Algorithm~\ref{alg:hsdc}, structured matrix $Q$ has the form
 $$Q = Q^{(0)}*_{\h}Q^{(1)}*_{\h}\cdots *_{\h}Q^{(l)}. $$
 $Q^{(i)} \in \R^{n\times n}$, $0 \leq i\leq l $, is a block-diagonal matrix with $2^{i}$ diagonal blocks, where each diagonal block is an orthogonal matrix of the form $[H_1 \hskip 2pt H_2]$, with $H_1, H_2$ orthonormal 
HODLR matrices computed in Line~\ref{honb} of Algorithm~\ref{alg:hsdc}. The computation of the eigenvectors matrix is completed by computing $Q^{(s)}$, a block-diagonal orthogonal matrix with $2^s$ orthogonal dense 
 diagonal blocks that are computed in Line~\ref{heig} of Algorithm~\ref{alg:hsdc}.    
  
The overall storage required to store $Q$ equals to the sum of memory requirements for matrices $Q^{(i)},   0\leq i \leq  s$. Assume that the off-diagonal ranks occurring in matrices 
$Q^{(i)}$, $ 0\leq i < s$, are bounded by $\tilde{k}$. To determine the storage, we use that $Q^{(i)}$, for $0\leq i<s$, has $2^i$ diagonal blocks of the form $[H_1 \hskip 2pt H_2]$, where the storage of both $H_1$ and $H_2$ requires 
$\calO(\tilde{k}\frac{n}{2^{i}}\log_2 \frac{n}{2^{i}})$ memory. Hence we get 
that the storage for matrices $Q^{(i)}$, $0\leq i<s$, adds up to 
\begin{align}
  \sum_{l=0}^{s-1} \tilde{k} 2^{l+1}\frac{n}{2^{l}}\log_2{\frac{n}{2^l}} &= 2\tilde{k} n\sum_{l = 0}^{s-1} \log_2 \frac{n}{2^l}
= 2\tilde{k} n\left(s\log_2 n - \frac{(s-1)s}{2} \right)\nonumber\\
&= \tilde{k} n \log^2_2 \frac{n}{n_{\stp}} + \log_2\frac{n}{n_{\stp}} (\log_2 n_{\stp} +1/2)\label{eq:storage_Q}\text{.}
 \end{align}
Moreover, the storage of $Q^{(s)}$ requires $2^s n^2_{\stp} = n n_{\stp}$ units of memory. Hence, from the latter and~\eqref{eq:storage_Q} follows that the 
overall memory needed for storing $Q$ is $\calO(\tilde{k} n\log^2n)$. 

\subsection{Computational complexity}
Now we derive the theoretical complexity of Algorithm~\ref{alg:hsdc}, based on the complexity of operations given in Table~\ref{table:complexity_HODLR}. The numerical results in 
section~\ref{sec:experiments} give an insight how the algorithm behaves in practice, and confirm theoretical results. 

We first note that for a HODLR matrix of size $m$ and rank $k$ the complexity of one divide step, computed in Line~\ref{hshift}--Line~\ref{hdc}, is $\calO(k^2m\log_2^2m)$. When 
performing level $l$ divide step, the computation involves 
$2^l$ HODLR matrices of size $n/2^l\times n/2^l$. Denoting with $\tilde{k}$ an upper bound for the off-diagonal ranks appearing in the process, similarly as in the previous section we
derive the complexity of our algorithm: 
\begin{align*}
  \sum_{l=0}^{s-1} \tilde{k}^2 2^{l}\frac{n}{2^{l}}\log_2^2{\frac{n}{2^l}} &= \tilde{k}^2 n\sum_{l = 0}^{s-1} \log_2^2 \frac{n}{2^l}
= \tilde{k}^2 n\left(s\log_2 n (\log_2 n_{\stp} +1) + \frac{(s-1)s(2s-1)}{6} \right)\nonumber\\
&\approx \calO(\tilde{k}^2 n \log^3_2 n)\label{eq:alg_complexity} \text{.}
 \end{align*}
At the final level of recursive application of Algorithm~\ref{alg:hsdc}, when the algorithm is applied to matrices of size not larger than $n_{\stp}$, the 
complexity comes from diagonalizing $2^s$ dense matrices, i.e., equals to $\calO(nn^2_{\stp})$. Thus the overall complexity of Algorithm~\ref{alg:hsdc} is $\calO(\tilde{k}^2 n\log^3n)$. 
\section{Numerical experiments}
\label{sec:experiments}

In this section, we show the performance of our \Matlab{} implementation of the spectral divide-and-conquer method in the HODLR format for various matrices. All 
computations were performed in \Matlab{} version R2016b on a machine with the dual Intel Core i7-5600U 2.60GHz CPU,  $256$ KByte of level 2 cache  and $12$ GByte of RAM.   

In order to draw a fair comparison with respect to highly optimized \Matlab{} built-in functions, all experiments were carried out on a single core. The memory requirements shown in 
Example~\ref{ex:scalability} are obtained experimentally, using \Matlab{} built-in functions. 

In all experiments, we set the truncation tolerance to $\epsilon = 10^{-10}$, the minimal block-size $n_{\min} = 250$ for tridiagonal matrices and $n_{\min} = 500$ for $b$-banded matrices with $b>1$.  Moreover, 
the stopping tolerance in the {\tt hDWH} algorithm is set to $\varepsilon = 10^{-15}$. In  Algorithm~\ref{alg:basis_complete} we use the oversampling parameter $p = 10$.  
 We use breakeven points in~\cite{KressnerSus2017} to set the termination criterion in Algorithm~\ref{alg:hsdc}. For tridiagonal matrices we use $n_{\stp} = 3250$, 
 for $2$-banded matrices $n_{\stp}= 1750$  and  $n_{\stp} = 2500$ for $b$-banded with $b > 2$.   

The efficiency of our algorithm is tested on a set of matrices coming from applications, as well as on various synthetic matrices. 

\subsection{Generation of test matrices}
\label{sec:matrix_gen}
To generate synthetic matrices, we employ the procedure explained 
in~\cite[section $6$]{KressnerSus2017} that uses a sequence of Givens rotations to obtain a symmetric banded matrix with a 
prescribed bandwidth and spectrum, starting from a diagonal matrix containing $n$ eigenvalues. As 
the accuracy of computed spectral projectors depends
on the relative spectral gap, we generate matrices such that $\gap$ is constant whenever the spectrum is split in half. We generate such a spectrum by first dividing the interval $[-1,1]$ into $[-1,-\gap]\cup [\gap, 1]$
and then recursively applying the same procedure to both subintervals. In particular, interval $[c,d]$ is split into $[c, \frac{c+d}{2} - \frac{d-c}{2}\gap]\cup [\frac{c+d}{2}+\frac{d-c}{2}\gap,d]$. 
The recursive division stops when the number of subintervals is $\leq n/n_{\stp}$. To each subinterval we assign equal number of eigenvalues coming from a uniform distribution. 
We observed similar results for eigenvalues coming from a geometric distribution, but 
we omit them to avoid redundancy. 

\begin{example}[\bfseries{Percentage and conditioning of selected columns}]
\label{ex:percentage_cond_sel_col} 
\rm We first investigate the percentage of selected columns throughout Algorithm~\ref{alg:hsdc} depending on a given threshold $\delta$, together with the condition number of the selected columns. We show results 
for matrices of size $n = 10240$, with bandwidths $b=1$ and $b=8$, and spectral gaps $\gap = 10^{-2}$ and $\gap = 10^{-6}$, generated as described above.  In 
this example we ensure that in all divide steps of 
Algorithm~\ref{alg:hsdc} the gap between separated parts of the spectrum corresponds to $\gap$, by computing the shift $\mu$ as the median of eigenvalues of a considered matrix. In each 
divide step in Algorithm~\ref{alg:hsdc} we compute 
the percentage of selected columns, and finally we show their average for each $\delta$. Moreover, we present a 
maximal condition number of the selected columns in the whole divide-and-conquer process for a given $\delta$. As expected, smaller values of $\delta$ lead to a 
higher percentage of selected columns, but this leads to a higher condition number as well.  Figure~\ref{fig:tridiag_perc_cond} and Figure~\ref{fig:band_8_perc_cond} show that already for $\delta \geq 0.4$
we get a good trade-off between the percentage of selected columns and the condition number. This also implies that the off-diagonal ranks in the eigenvectors matrix remain low.   

\begin{figure}[tbhp]
\begin{center}
\includegraphics[width=0.45\textwidth]{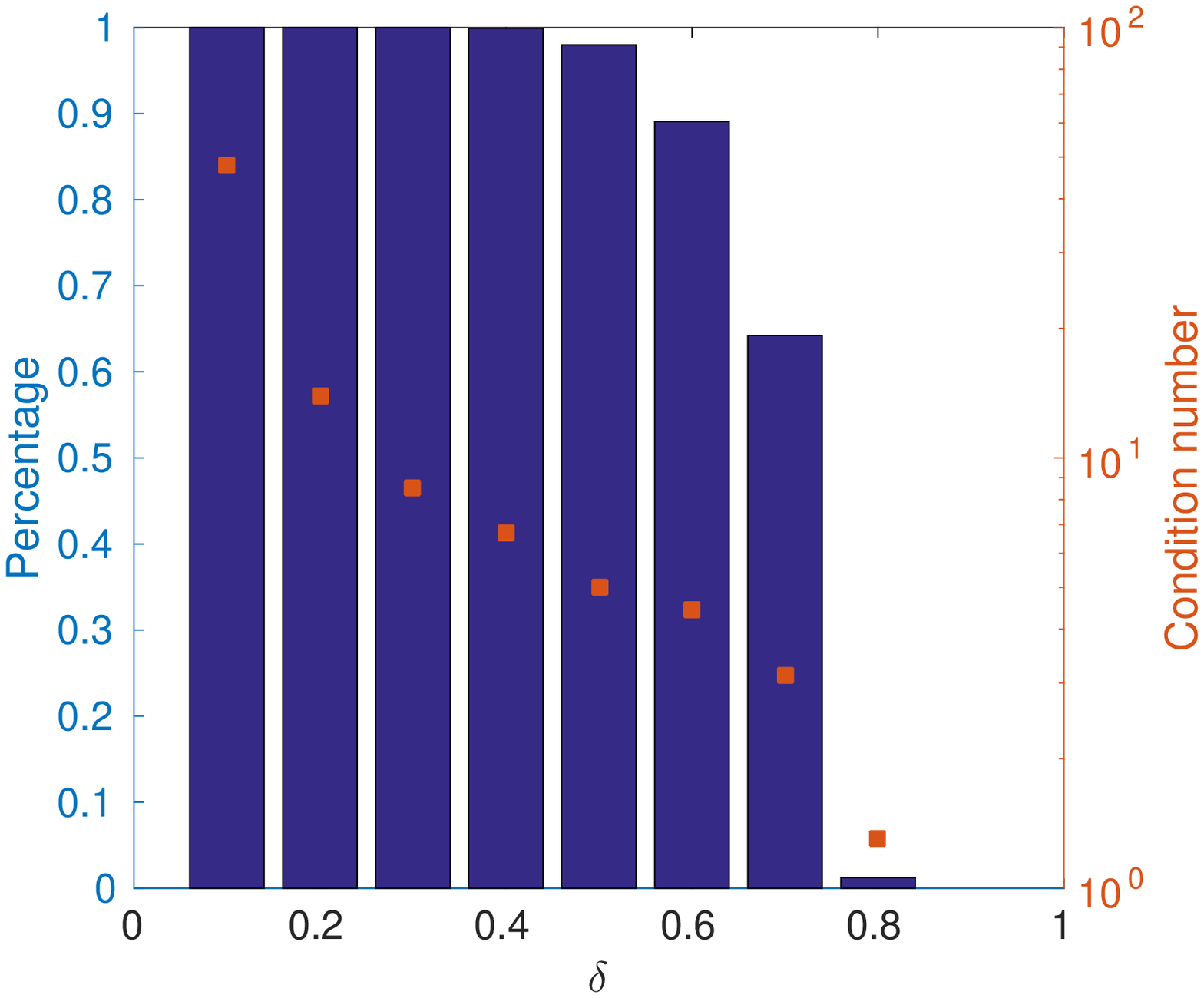}
\includegraphics[width=0.45\textwidth]{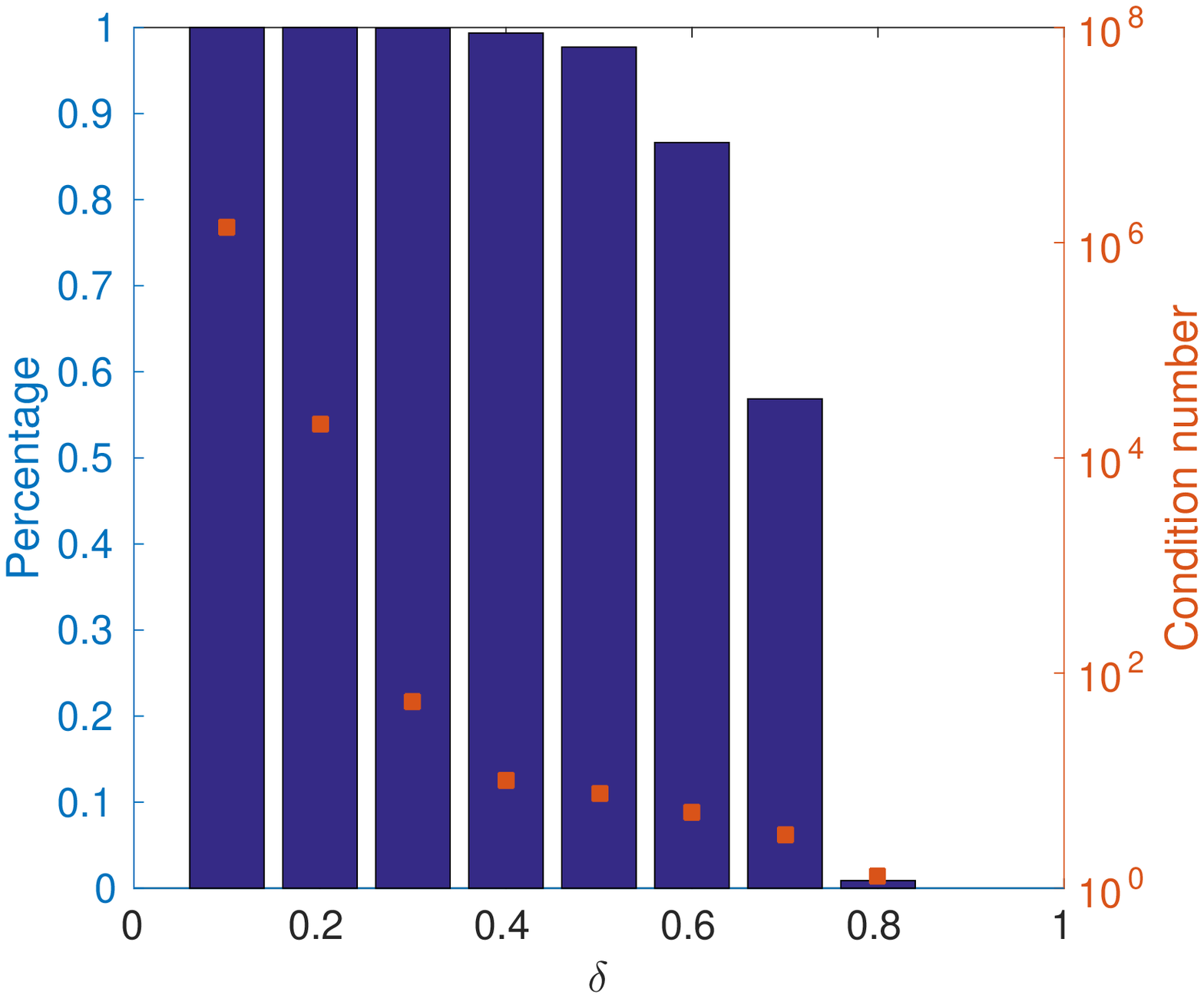}
\end{center}
\caption{Example~\ref{ex:percentage_cond_sel_col}. Percentage of selected columns and their condition number for a tridiagonal matrix with eigenvalues in $[-1,1]$ with relative spectral gap
$\gap = 10^{-2}$ (left) and $\gap = 10^{-6}$ (right).}
\label{fig:tridiag_perc_cond}
\end{figure}

\begin{figure}[tbhp]
\begin{center}
\includegraphics[width=0.45\textwidth]{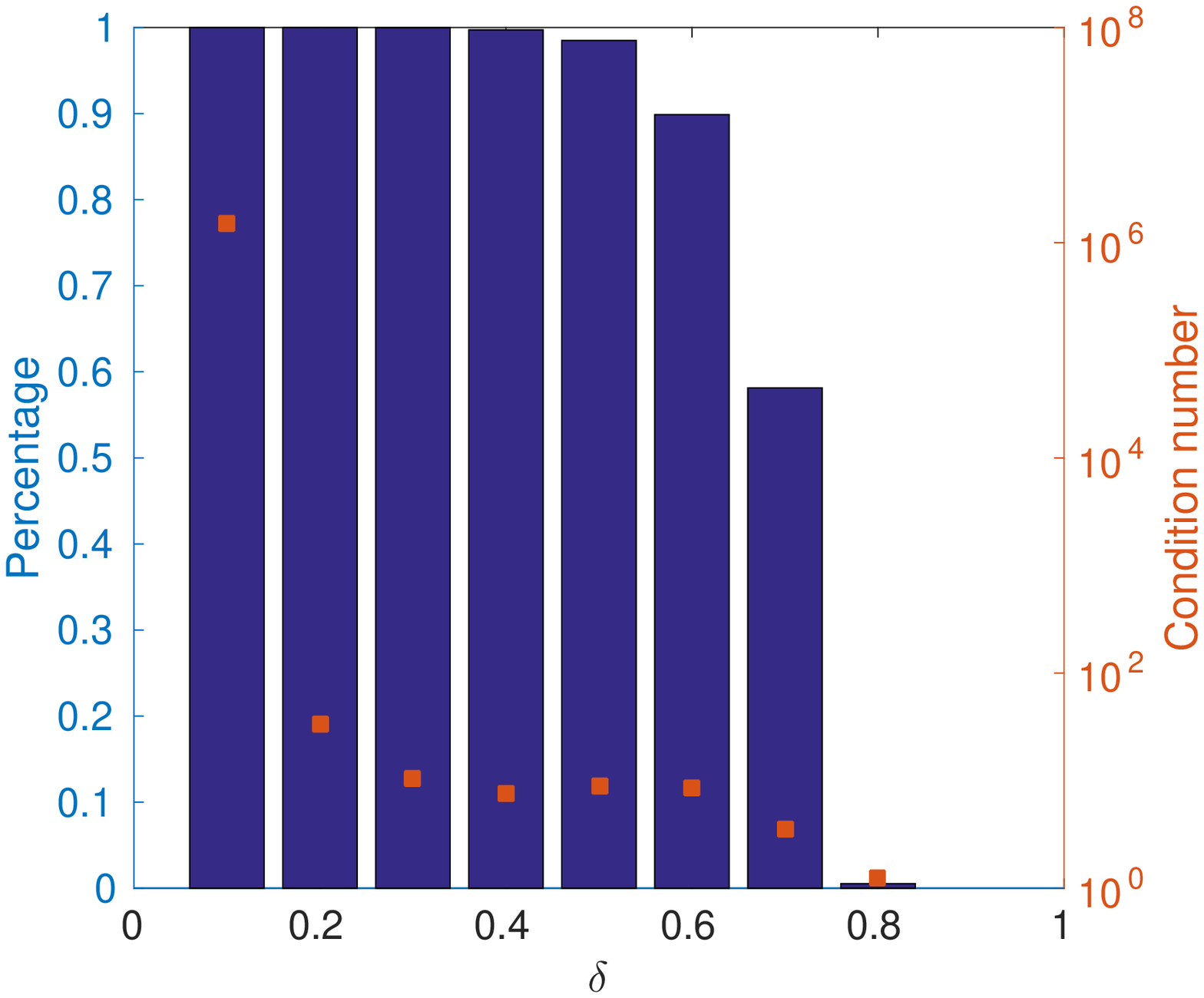}
\includegraphics[width=0.45\textwidth]{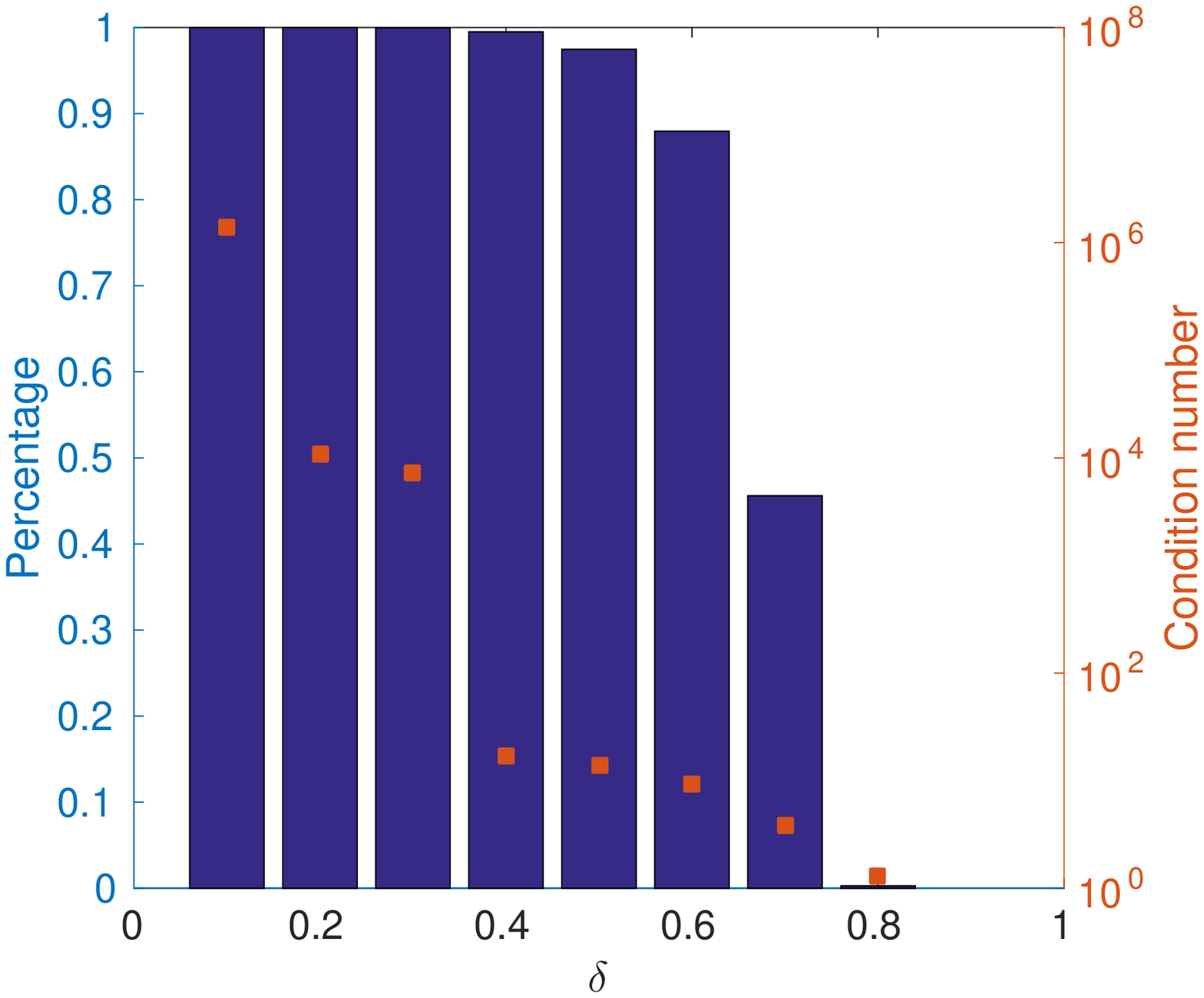}
\end{center}
\caption{Example~\ref{ex:percentage_cond_sel_col}. Percentage of selected columns and their condition number for a $8$-banded matrix with eigenvalues in $[-1,1]$ with relative spectral gap
$\gap = 10^{-2}$ (left) and $\gap = 10^{-6}$ (right).}
\label{fig:band_8_perc_cond}
\end{figure}
\end{example}

\begin{example}[\bfseries{Breakeven point relative to {\tt eig}}]
\label{ex:breakeven_point} 
\rm 
Our runtime comparisons are performed on generated $n\times n$ banded matrices.  We examine for which values of $n$ Algorithm~\ref{alg:hsdc} outperforms {\tt eig}.  In Table~\ref{table:break_even_point} we show
breakeven points for banded matrices constructed as in section~\ref{sec:matrix_gen}, with $\gap \in \lbrace 10^{-1},10^{-2}, 10^{-3}, 10^{-4}\rbrace$. We use the threshold parameter 
$\delta  = 0.4$.  
For $b=2$ and $b=4$ banded matrices our algorithm becomes faster than {\tt eig} for relatively small  $n$. This is due to the fact that \Matlab{}'s {\tt eig} first performs tridiagonal 
reduction. Our results show the benefit of avoiding the reduction of a banded matrix to a tridiagonal form, especially when the bandwidth is small.  
However, for tridiagonal matrices the breakeven point is relatively high. 
\begin{table}[ht!]
\caption{Breakeven point of hSDC relative to {\tt eig} applied for banded matrices with various bandwidths and spectral gaps.}
\label{table:break_even_point}
\centering
{\renewcommand{\arraystretch}{1}
\begin{tabular}{c||c||c||c||c}
\backslashbox[5mm]{gap}{b} &$1$ &$2$ &$4$ &$8$ \\
 \hline
 \hline
 $10^{-1}$ &$n = 10000 $ &$n = 2100$ &$n= 3000$ &$n=5300$ \\
 \hline
 $10^{-2}$ &$n = 14000 $ &$n = 2500$ &$n= 3800$ &$n=7400$  \\
 \hline
 $10^{-3}$ &$n = 16000$ &$n = 2800$ &$n= 4900$ &$n=8000$  \\
 \hline
 $10^{-4}$ &$n = 17000 $ &$n = 3000$ &$n= 5200$ &$n=8600$  \\
\end{tabular}
 }
\end{table}
\end{example}

\begin{example}[\bfseries{Accuracy for various matrices}]
\label{ex:accuracy} 
\rm In this example, we test the accuracy of the computed spectral decomposition. Denoting with $Q = [q_1, \ldots, q_n]$ and $\Lambda = \diag(\lambda_1, \ldots, \lambda_n)$ the output of Algorithm~\ref{alg:hsdc}, and 
with  $\tilde{Q} = [\tilde{q}_1, \ldots, \tilde{q}_n]$ and $\tilde{\Lambda}  = \diag(\tilde{\lambda}_1, \ldots, \tilde{\lambda}_n)$ 
the eigenvalue decomposition obtained using \Matlab{}'s {\tt eig}, we consider four different error metrics: 
\begin{itemize}
 \item  the largest relative error in the computed eigenvalues: $e_{\lambda} =  \underset{i}{\max}\hskip 3pt \vert \lambda_i - \tilde{\lambda}_i\vert/\Vert A\Vert_2, $ 
\item  the largest relative residual norm: $e_{\operatorname{res}} = \underset{i}{\max}\hskip 3pt \Vert Aq_i - \lambda_i q_i\Vert_2/\Vert A\Vert_2,$ 
\item  the loss of orthogonality: $ e_{\operatorname{orth}} =  \underset{i}{\max} \Vert Q^T q_i - e_i\Vert_2,$ 
\item the largest error in the computed eigenvectors : $e_{Q} =  \underset{i}{\max}\hskip 3pt \vert 1 - \cos \measuredangle (q_i,\tilde{q}_i) \vert$. 
\end{itemize}

In the subsequent experiments we set $\delta = 0.4$. 
 
\begin{enumerate} 
 \item First we show the accuracy of the newly proposed algorithm for tridiagonal matrices. For matrices of size smaller than $3250$,  we use $n_{\stp} = 500$, which allows us to perform at least one divide step 
 in Algorithm~\ref{alg:hsdc}.  We consider some of the matrices suggested in~\cite{MarqVomeDemmParl2009}: 
 \begin{itemize}
\item the BCSSTRUC1 set in the Harwell-Boeing Collection~\cite{Davis2007}. Considered problems are in fact generalized eigenvalue problems, with $M$ a mass matrix and  $K$ a stiffnes matrix. Each problem 
is transformed  into an equivalent standard eigenvalue problem  $L^{-1}KL^{-T}x = \lambda x$, where $L$ denotes the Cholesky factor of 
$M$. Finally, matrices are reduced to tridiagonal form via \Matlab{} function {\tt hess}.
\item The symmetric Alemdar and Cannizzo matrices, and matrices from the NASA set~\cite{Davis2007}. Considered matrices are reduced to tridiagonal form using \Matlab{} function {\tt hess}.
\item The $(1,2,1)$ symmetric tridiagonal Toeplitz matrix.
\item The Legendre-type tridiagonal matrix. 
\item The Laguerre-type tridiagonal matrix. 
\item The Hermite-type tridiagonal matrix.
\item Symmetric tridiagonal matrices with eigenvalues coming from a random $(0,1)$ distribution and a uniform distribution on $[-1,1]$.
\end{itemize} 

In Table~\ref{table:error_test_matrices} we report the observed accuracies. The results are satisfactory, and the errors are roughly of order of the truncation tolerance $\epsilon = 10^{-10}$. We also mention 
that the percentage of selected columns, as well as the condition number of selected columns throughout Algorithm~\ref{alg:hsdc} were along the lines the results presented in Example~\ref{ex:percentage_cond_sel_col}.

\begin{table}[tbhp]

\caption{Accuracy of hSDC for tridiagonal matrices from Example~\ref{ex:accuracy}.}
 \centering
{\renewcommand{\arraystretch}{1}
\begin{tabular}{c||c||c|c|c|c|c}
&matrix &$n$ &$e_{\lambda}$ &$e_{\operatorname{res}}$ & $e_{\operatorname{orth}}$ &$e_{Q}$    \\
 \hline
 \hline
  \multirow{3}{*}{\begin{sideways}\tiny{BCSSTRUC1}\end{sideways}} &\text{bcsst08} &$1074$ &$4.4\cdot 10^{-14}$ &$ 2.2\cdot 10^{-12}$ &$5.6\cdot 10^{-10}$ &$ 5.6\cdot 10^{-12}$\\
 \cline{2-7}
 &\text{bcsst09} &$1083$ &$5.2\cdot 10^{-11}$ &$2.3\cdot 10^{-11}$ &$7.8\cdot 10^{-10}$ &$6.3\cdot 10^{-12}$\\
 \cline{2-7}
 &\text{bcsst11} &$1474$ &$1.8\cdot 10^{-11}$ &$1.4\cdot 10^{-10}$ &$2.6\cdot 10^{-9}$ &$7.3\cdot 10^{-11}$\\
 \cline{2-7}
 \hline
 \hline
   \multirow{4}{*}{\begin{sideways}\tiny{NASA}\end{sideways}} &\text{nasa1824} &$1824$ &$3.2\cdot 10^{-12}$ &$ 6.8\cdot 10^{-9}$ &$2.5\cdot 10^{-9}$ &$ 1.5\cdot 10^{-10}$\\
 \cline{2-7}
 &\text{nasa2146} &$2146$ &$1.5\cdot 10^{-10}$ &$2.9\cdot 10^{-9}$ &$1.8\cdot 10^{-9}$ &$7.1\cdot 10^{-11}$\\
 \cline{2-7}
 &\text{nasa2190} &$2190$ &$1.5\cdot 10^{-12}$ &$2.1\cdot 10^{-12}$ &$6.7\cdot 10^{-10}$ &$4.1\cdot 10^{-11}$\\
 \cline{2-7}
 &\text{nasa4704} &$4704$ &$8.9\cdot 10^{-12}$ &$2.9\cdot 10^{-10}$ &$9.8\cdot 10^{-9}$   &$ 5.3\cdot 10^{-10}$\\
 \hline
 \hline
 &\text{Cannizzo matrix} &$4098$ &$ 3.4\cdot 10^{-11}$ &$ 8.3\cdot 10^{-10}$ &$1.5\cdot 10^{-9}$ &$2.6\cdot 10^{-10}$\\
 \cline{2-7}
 &\text{Alemdar matrix} &$6245$  &$1.25\cdot 10^{-9}$ &$7.4\cdot 10^{-8} $ &$2.5\cdot 10^{-9}$ &$ 2.8\cdot 10^{-10}$\\
 \hline
 \hline
 &\text{(1,2,1) matrix} &$10000$ &$ 2.4\cdot 10^{-10}$ &$1.6\cdot 10^{-9}$  &$2.5\cdot 10^{-11} $ &$1.3\cdot 10^{-11}$ \\
\cline{2-7}
  \cline{2-7}
   &\text{Clement-type} &$10000$ &$1.75\cdot 10^{-10}$ &$2.3\cdot 10^{-9}$ &$2.6\cdot 10^{-9}$ &$1.4\cdot 10^{-10}$\\
  \hline
  \hline
   &\text{Legendre-type} &$10000$ &$2.4\cdot 10^{-11}$ &$2.7\cdot 10^{-10}$ &$9.6\cdot 10^{-11}$ &$2.7\cdot 10^{-11}$\\
  \cline{2-7}
   &\text{Laguerre-type} &$10000$ &$3.4\cdot 10^{-11}$ &$1.9\cdot 10^{-10}$ &$3.7\cdot 10^{-9}$ &$4.3\cdot 10^{-11}$\\
  \cline{2-7}
   &\text{Hermite-type} &$10000$ &$9.8\cdot 10^{-11}$ &$3.1\cdot 10^{-9}$ &$9.3\cdot 10^{-10}$ &$3.2\cdot 10^{-11}$\\
 \hline
 \hline
 &\text{ Random normal $(0,1)$} &$10000$ &$6.9\cdot10^{-10}$ &$4.5\cdot 10^{-8}$ &$ 7.5\cdot 10^{-9}$ &$1.1\cdot 10^{-10} $ \\
 \cline{2-7}
  &\text{Uniform in $[-1,1]$} &$10000$ &$1.1\cdot 10^{-11}$ &$2.1\cdot 10^{-9}$ &$3.5\cdot 10^{-10}$ &$1.8\cdot 10^{-11}$\\
  \hline
 \end{tabular}
 }
  \label{table:error_test_matrices}
\end{table}

\item Now we examine the dependency of the error measures on the decreasing spectral gap. We construct $8$-banded matrices of size $n = 10240$ with 
$\gap = 10^{-i}, i = 1,\ldots, 10$ using a method from section~\ref{sec:matrix_gen}. As in Example~\ref{ex:percentage_cond_sel_col}, we ensure that the gap between two separated parts of spectrum in all divide steps of Algorithm~\ref{alg:hsdc} is the 
prescribed $\gap$. Figure~\ref{fig:error_vs_gap_tridiag} shows that our algorithm 
preforms well for matrices with larger spectral gaps, and confirms the expected behaviour of errors. The error growth is a result of the decreasing accuracy when computing
spectral projectors associated with decreasing spectral gaps. For gaps of order $\leq 10^{-9}$ the algorithm breaks down due to indefinitness
of a matrix in Line~\ref{alg_chol1} of Algorithm~\ref{alg:hdwh}. 

\begin{figure}[tbhp]
\begin{center}
\includegraphics[width=0.45\textwidth]{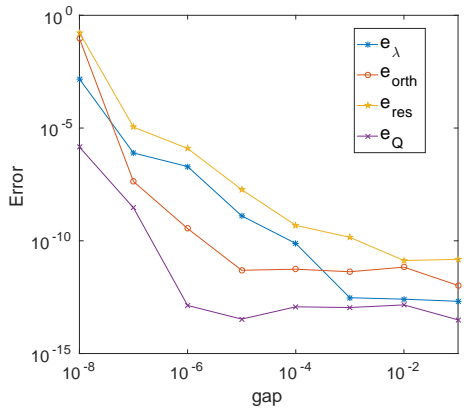}
\end{center}
\caption{Example~\ref{ex:accuracy}. Behaviour of the errors with respect to a decreasing spectral gap for banded matrices.}
\label{fig:error_vs_gap_tridiag}
\end{figure}

\end{enumerate}
\end{example}

\begin{example}[\bfseries{Scalability}]
\label{ex:scalability} 
\rm For $n\times n$ tridiagonal matrices, generated as in section~\ref{sec:matrix_gen} with $\gap = 10^{-2}$, we demonstrate the performance of our algorithm with respect to $n$. Again we use the 
threshold parameter $\delta = 0.4$.  We show that the asymptotic behaviour of our algorithm matches
the theoretical bounds both for the computational time and storage requirements. Figure~\ref{fig:scale_tridiag_1e-2} (left) shows that time needed to compute the complete
spectral decomposition follows the expected $\calO(n\log^3(n))$ reference line, whereas Figure~\ref{fig:scale_tridiag_1e-2} (right) demonstrates that the memory required to store the 
matrix of eigenvectors is $\calO(n\log^2(n))$. 

\begin{figure}[tbhp]
\begin{center}
\includegraphics[width=0.45\textwidth]{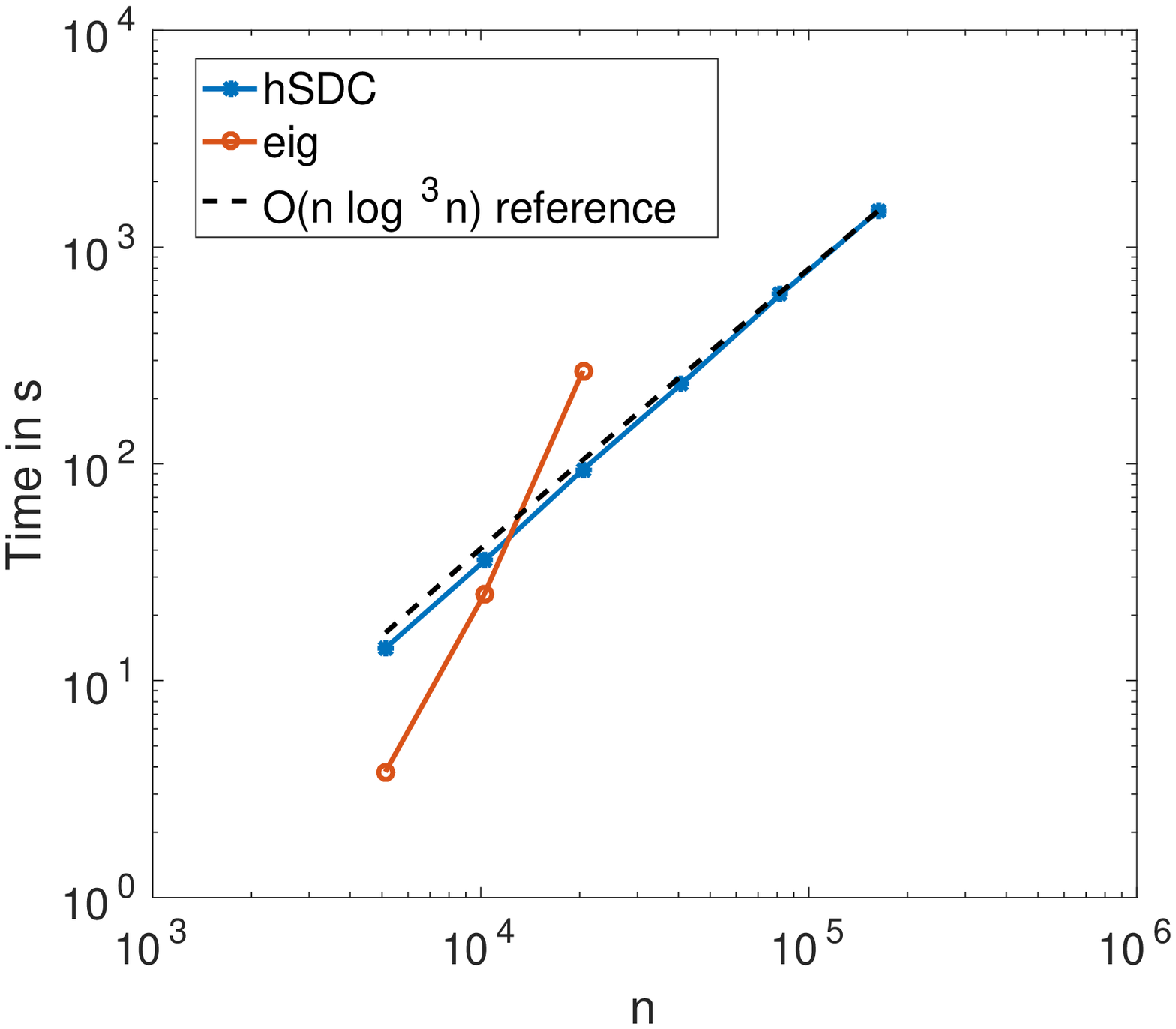}
\includegraphics[width=0.45\textwidth]{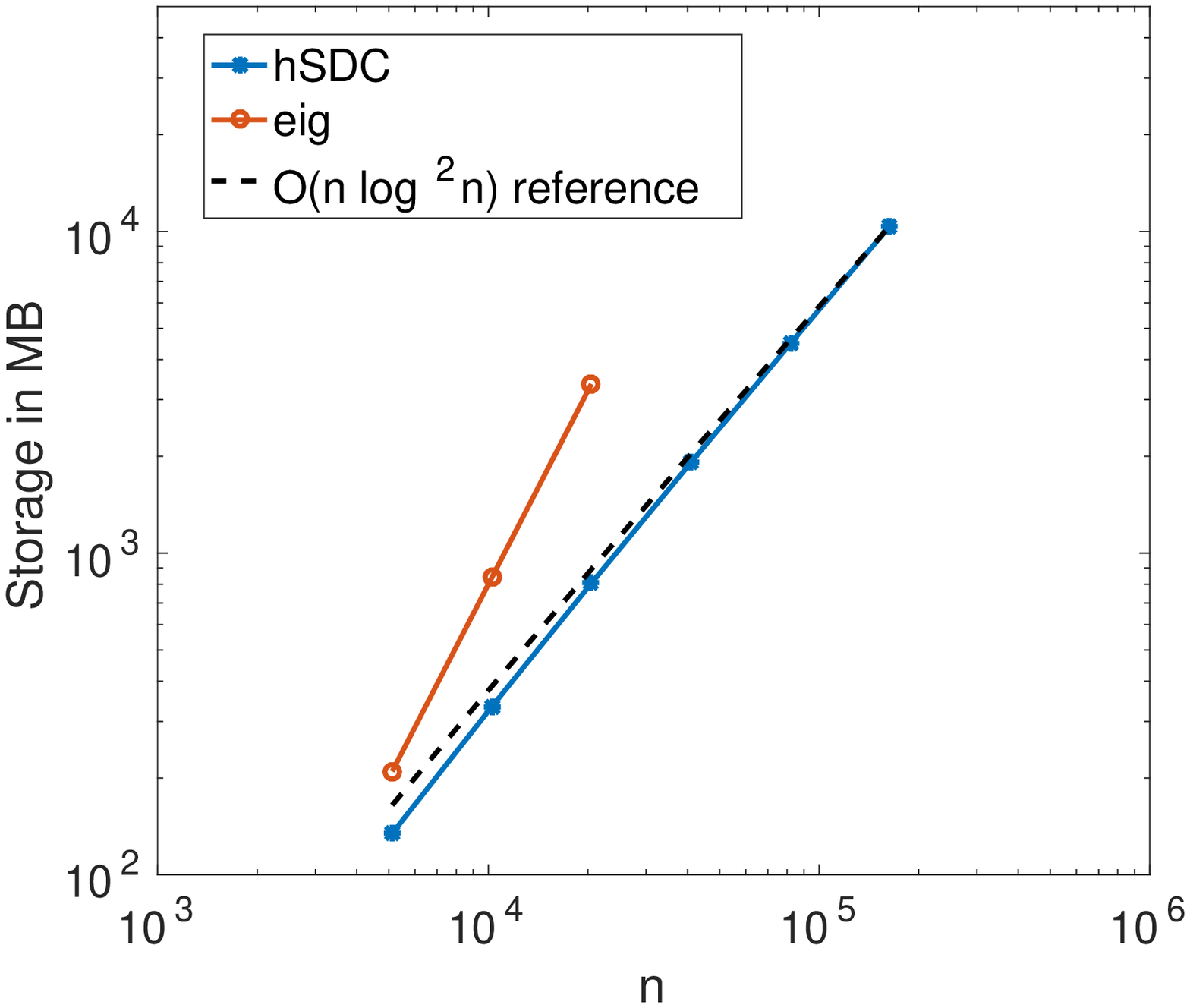}
\end{center}
\caption{Example~\ref{ex:scalability}. Performance of the hSDC algorithm with respect to $n$ for tridiagonal matrices. Left: Computational time. Right: Memory requirements.}
\label{fig:scale_tridiag_1e-2}
\end{figure}

\end{example}

\section{Conclusion}
\label{sec:conclusion}

In this work we have proposed a new fast spectral divide-and-conquer algorithm for computing the complete spectral decomposition of symmetric banded matrices. The algorithm exploits the fact that 
spectral projectors of banded matrices can be efficiently computed and stored in the HODLR format. We have presented a fast novel method for 
selecting well-conditioned columns of a spectral projector based on a Cholesky decomposition with pivoting, and provided a theoretical justification for the method. This method
enables us to efficiently split the computation of the spectral decomposition of a symmetric HODLR matrix into two smaller subproblems.  

The new spectral D\&C method is implemented in the HODLR format and  has a linear-polylogarithmic complexity. In the numerical experiments, performed both on 
synthetic matrices and matrices coming from applications, we have verified the efficiency of our method, and have shown that it is a competitive alternative to the state-of-the-art methods for some classes 
of banded matrices. 

\begin{paragraph}{Acknowledgements.}
The authors would like to thank Stefano Massei for helpful discussions on this paper.
\end{paragraph}

\bibliographystyle{plain}
\bibliography{biblio}

\begin{thebibliography}{10}

\bibitem{Ambikasaran2013}
S.~Ambikasaran and E.~Darve.
\newblock An {$\mathcal{O}(N\log N)$} fast direct solver for partial
  hierarchically semi-separable matrices: with application to radial basis
  function interpolation.
\newblock {\em J. Sci. Comput.}, 57(3):477--501, 2013.

\bibitem{Arbenz1992}
P.~Arbenz.
\newblock Divide and conquer algorithms for the bandsymmetric eigenvalue
  problem.
\newblock {\em Parallel Comput.}, 18(10):1105--1128, 1992.

\bibitem{AuckBlumBungHuck2011}
T.~Auckenthaler, V.~Blum, H.-J. Bungartz, T.~Huckle, R.~Johanni, L.~Kr{\"a}mer,
  B.~Lang, H.~Lederer, and P.~R. Willems.
\newblock Parallel solution of partial symmetric eigenvalue problems from
  electronic structure calculations.
\newblock {\em Parallel Computing}, 37(12):783--794, 2011.

\bibitem{Ballani2016}
J.~Ballani and D.~Kressner.
\newblock {\em Matrices with hierarchical low-rank structure}.
\newblock {CIME} {S}ummer {S}chool on {E}xploiting {H}idden {S}tructure in
  {M}atrix {C}omputations, 2016.

\bibitem{BienIgualKressPet2011}
P.~Bientinesi, F.~D Igual, D.~Kressner, M.~Petschow, and E.~S.
  Quintana-Ort{\'\i}.
\newblock Condensed forms for the symmetric eigenvalue problem on
  multi-threaded architectures.
\newblock {\em Concurrency and Computation: Practice and Experience},
  23(7):694--707, 2011.

\bibitem{BiscLangSun2000}
C.~H. Bischof, B.~Lang, and X.~Sun.
\newblock A framework for symmetric band reduction.
\newblock {\em ACM Trans. Math. Software}, 26(4):581--601, 2000.

\bibitem{Cuppen1980/81}
J.~J.~M. Cuppen.
\newblock A divide and conquer method for the symmetric tridiagonal
  eigenproblem.
\newblock {\em Numer. Math.}, 36(2):177--195, 1980/81.

\bibitem{Davis2007}
T.~A. Davis and Y.~Hu.
\newblock The {U}niversity of {F}lorida {S}parse {M}atrix {C}ollection.
\newblock {\em ACM Trans. Math. Software}, 38(1):1--25, 2011.

\bibitem{GolVanL2013}
G.~H. Golub and C.~F. Van~Loan.
\newblock {\em Matrix computations}.
\newblock Johns Hopkins University Press, Baltimore, MD, fourth edition, 2013.

\bibitem{Gu1995}
Ming Gu and Stanley~C. Eisenstat.
\newblock A divide-and-conquer algorithm for the symmetric tridiagonal
  eigenproblem.
\newblock {\em SIAM J. Matrix Anal. Appl.}, 16(1):172--191, 1995.

\bibitem{Hackbusch2015}
W.~Hackbusch.
\newblock {\em Hierarchical matrices: algorithms and analysis}, volume~49 of
  {\em Springer Series in Computational Mathematics}.
\newblock Springer, Heidelberg, 2015.

\bibitem{HaidLtaiDong2011}
A.~Haidar, H.~Ltaief, and J.~Dongarra.
\newblock Parallel {R}eduction to {C}ondensed {F}orms for {S}ymmetric
  {E}igenvalue {P}roblems {U}sing {A}ggregated {F}ine-grained and
  {M}emory-aware {K}ernels.
\newblock In {\em Proceedings of 2011 International Conference for High
  Performance Computing, Networking, Storage and Analysis}, pages 8:1--8:11.
  ACM, 2011.

\bibitem{HaidLtaiDong2012}
A.~Haidar, H.~Ltaief, and J.~Dongarra.
\newblock Toward a high performance tile divide and conquer algorithm for the
  dense symmetric eigenvalue problem.
\newblock {\em SIAM J. Sci. Comput.}, 34(6):C249--C274, 2012.

\bibitem{HaidSolcGates2013}
A.~Haidar, R.~Solc{\`a}, M.~Gates, S.~Tomov, T.~Schulthess, and J.~Dongarra.
\newblock Leading {E}dge {H}ybrid {M}ulti-{GPU} {A}lgorithms for {G}eneralized
  {E}igenproblems in {E}lectronic {S}tructure {C}alculations.
\newblock In {\em Supercomputing}, volume 7905 of {\em Lecture {N}otes in
  {C}omputer {S}cience}, pages 67--80. Springer Berlin Heidelberg, 2013.

\bibitem{HalkoMartinsson2011}
N.~Halko, P.~G. Martinsson, and J.~A. Tropp.
\newblock Finding structure with randomness: probabilistic algorithms for
  constructing approximate matrix decompositions.
\newblock {\em SIAM Rev.}, 53(2):217--288, 2011.

\bibitem{Higham1996}
N.~J. Higham.
\newblock {\em Accuracy and stability of numerical algorithms}.
\newblock SIAM, Philadelphia, PA, 1996.

\bibitem{KressnerSus2017}
D.~Kressner and A.~{\v S}u{\v s}njara.
\newblock Fast computation of spectral projectors of banded batrices.
\newblock {\em SIAM J. Matrix Anal. Appl.}, 38(3):984--1009, 2017.

\bibitem{Lintner2002}
M.~Lintner.
\newblock {\em L{\"o}sung der 2D Wellengleichung mittels hierarchischer
  Matrizen}.
\newblock Doctoral thesis, TU M{\"u}nchen, 2002.

\bibitem{MarqVomeDemmParl2009}
O.~A. Marques, C.~V{\"o}mel, J.~W. Demmel, and B.~N. Parlett.
\newblock Algorithm 880: a testing infrastructure for symmetric tridiagonal
  eigensolvers.
\newblock {\em ACM Trans. Math. Software}, 35(1):Art. 8, 13, 2009.

\bibitem{MolerStew1978}
C.~B. Moler and G.~W. Stewart.
\newblock On the {H}ouseholder-{F}ox algorithm for decomposing a projection.
\newblock {\em J. Comput. Phys.}, 28(1):82--91, 1978.

\bibitem{NakaBaiGygi2010}
Y.~Nakatsukasa, Z.~Bai, and F.~Gygi.
\newblock Optimizing {H}alley's iteration for computing the matrix polar
  decomposition.
\newblock {\em SIAM J. Matrix Anal. Appl.}, 31(5):2700--2720, 2010.

\bibitem{NakaHigh2013}
Y.~Nakatsukasa and N.~J. Higham.
\newblock Stable and efficient spectral divide and conquer algorithms for the
  symmetric eigenvalue decomposition and the {SVD}.
\newblock {\em SIAM J. Sci. Comput.}, 35(3):A1325--A1349, 2013.

\bibitem{SoloBallDemmHoef016}
E.~Solomonik, G.~Ballard, J.~Demmel, and T.~Hoefler.
\newblock A communication-avoiding parallel algorithm for the symmetric
  eigenvalue problem.
\newblock arXiv:1604.03703, 2016.

\bibitem{VogelXiaCauBal2016}
J.~Vogel, J.~Xia, S.~Cauley, and V.~Balakrishnan.
\newblock Superfast divide-and-conquer method and perturbation analysis for
  structured eigenvalue solutions.
\newblock {\em SIAM J. Sci. Comput.}, 38(3):A1358--A1382, 2016.

\end{thebibliography}

\end{document}